\documentclass[12pt]{article}

\usepackage{lmodern}
\usepackage[T1]{fontenc}
\usepackage{enumitem}
\usepackage{graphicx}
\usepackage{tikz}
\usepackage{amsthm, amssymb, amsmath, amsfonts}
\usepackage{bbm}
\usepackage{mathtools}
\usepackage{mathrsfs}
\usepackage{bm}
\usepackage{xcolor}
\usepackage{caption}
\usepackage{titling}
\numberwithin{equation}{section}
\numberwithin{figure}{section}

\theoremstyle{plain}
\newtheorem{theorem}{Theorem}[section] 
\newtheorem{corollary}[theorem]{Corollary} 
\newtheorem{lemma}[theorem]{Lemma} 

\theoremstyle{definition}
\newtheorem{definition}[theorem]{Definition} 
\newtheorem{notation}[theorem]{Notation}
\newtheorem*{acknowledgements*}{Acknowledgements}
\newtheorem{remark}[theorem]{Remark} 

\theoremstyle{remark}

\newcommand{\N}{\mathbb{N}}
\newcommand{\Z}{\mathbb{Z}}

\newcommand{\R}{\mathbb{R}}
\newcommand{\C}{\mathbb{C}}

\newcommand{\esp}{\mathbb{E}}
\newcommand{\sN}{{\scriptscriptstyle N}}

\DeclareMathOperator{\tr}{tr}
\newcommand{\1}{\mathbbm{1}}
\newcommand{\dd}{\mathop{}\!\mathrm{d}}

\newcommand{\PP}{\mathcal{P}}
\newcommand{\NC}{\mathcal{NC}}
\DeclareMathOperator{\Kr}{Kr}

\newcommand{\DT}{\mathcal{DT}}
\newcommand{\U}{\mathcal{U}}

\newcommand{\cA}{\mathcal{A}}
\let\phi=\varphi

\DeclarePairedDelimiter{\abs}{\lvert}{\rvert}

\DeclarePairedDelimiter{\set}{\{}{\}}

\DeclarePairedDelimiter{\paren}{(}{)}
\DeclarePairedDelimiter{\bracket}{[}{]}

\newcommand{\xx}[2]{\relax}
\newcommand{\thebottomline}{\renewcommand{\thefootnote}{}
  \renewcommand{\footnoterule}{}
  \phantom{M}\footnotetext{\tiny{}\textit{\jobname.tex}\hfill
    \textit{\noindent\romannumeral\day.%
\romannumeral\month.\expandafter\xx\romannumeral\year}}}

\makeatletter
\renewcommand\section{\@startsection {section}{1}{\z@}%
                                   {-3.5ex \@plus -1ex \@minus -.2ex}%
                                   {2.3ex \@plus.2ex}%
                                   {\normalsize\bf}}
\renewcommand\subsection{\@startsection {subsection}{1}{\z@}%
                                   {-3.5ex \@plus -1ex \@minus -.2ex}%
                                   {2.3ex \@plus.2ex}%
                                   {\normalsize\bf}}                                   
\long\def\@makefntext#1{\@setpar{\@@par\@tempdima
\hsize \advance\@tempdima-10pt\parshape \@ne
10pt\@tempdima}\par \parindent 1em\noindent \hbox to
\z@{\hss$\m@th^{\@thefnmark}$}#1} 
\makeatother

\renewcommand{\thefootnote}{(\arabic{footnote})}
\newcommand{\ab}{\allowbreak}

\newcommand{\re}{\mathrm{Re}}
\newcommand{\cP}{\mathcal{P}}
\allowdisplaybreaks

    \title{\scshape\fontsize{16}{19}\selectfont 
     Infinitesimal Freeness of Wigner Matrices}
    
    \author{\scshape\fontsize{14}{17}\selectfont
      Samuel Gurrola-Viramontes and James A. Mingo}

\date{}

\begin{document}

\maketitle

\begin{abstract}
    In this paper, within the framework of real infinitesimal free probability introduced by Cébron and
    the second author, we compute the real infinitesimal free cumulants of independent complex Wigner matrices. Our approach relies on establishing a combinatorial relation between annular non-crossing partitions and families of directed graphs. As a consequence, we demonstrate that independent complex Wigner matrices are asymptotically real infinitesimally free. In particular, we show (under mild conditions) that a complex Wigner matrix is asymptotically infinitesimally free from its transpose.
\end{abstract}


\section{Introduction}

Since Wigner's initial work \cite{w1, w2}, there have been many extensions and generalizations of his results on limit laws for Hermitian random matrix ensembles. Wigner used the method of moments to show that the semi-circle law was the leading term in the $1/N$ expansion of the eigenvalue distribution of what is now called a Wigner matrix. Later authors considered unitarily invariant matrices given by a polynomial potential. Assuming some technical conditions on the potential, Johansson \cite[Thm. 2.4]{j} found the $1/N$ correction to the expansion. Johansson's formula had two parts. We will show here that the first can easily be expressed as the interaction between the matrix and its transpose; the second term contained a potential which is absent in the Wigner case. Later, Bai and Yao considered the corresponding question for Wigner matrices (see \cite[Eq. (1.1)]{by}). They showed that the (random) empirical infinitesimal process converged to a Gaussian process and found a formula for covariances of linear spectral statistics. Subsequently in \cite{em}, Enriquez and Ménard gave the $1/N$ correction for a certain class of Wigner matrices. We shall show that their description of the limiting infinitesimal distribution can be easily expressed in terms of the infinitesimal free cumulants, which were introduced in the complex case in \cite{fn} and in the real case in \cite{cm}. More precisely, the infinitesimal law of Wigner matrices obtained by Enriquez and Ménard\footnote{We have to correct a small error in their paper and let $X^{(\sN)} = \frac{1}{\sigma \sqrt{N}.} (W_{ij})$} can be expressed as follows
\begin{gather*}
    \frac{\1_{\set{r^2 = \sigma^2}}}{2} D_{\sigma}\paren*{\nu_1 - \nu_2}\ +\
    \frac{1}{2}\paren*{\frac{\alpha}{\sigma^4} - (2+\1_{\set{r^2 = \sigma^2}})}\paren*{\frac{x^4}{\sigma^4} - 4\frac{x^2}{\sigma^2} + 2}D_{\sigma}\nu_2\\
    + \frac{1}{2}\paren*{\frac{s^2}{\sigma^2} - (1+\1_{\set{r^2 = \sigma^2}})}\paren*{\frac{x^2}{\sigma^2} - 2}D_\sigma\nu_2,
\end{gather*}
where $r^2, \sigma^2, s^2, \alpha$ are parameters associated to the family of Wigner matrices (see Definition~\ref{def:wigner matrix} and Remark~\ref{rmk:wigner matrix}) and
\[
    \nu_1 = \frac{1}{2}\paren*{\delta_{-2} + \delta_2}, \qquad \nu_2(\dd x) = \frac{1}{\pi \sqrt{4 - x^2}} \, \1_{[-2, 2]}(x) \dd x.
\]
If we let $G$ be the Cauchy transform of the semi-circle law with variance $1$, $F = 1/G$ and $R_i(z) = \kappa_2' z + \kappa_4' z^3$, then the Cauchy transform of the infinitesimal law of Enriquez-Ménard (above) is
\begin{equation}\label{eq:cauchy and r transform}
g(z) = R_i(G(z)) G'(z) - F''(z)/(2 F'(z))
\end{equation}
where $\kappa_2' = s^2/\sigma^2 - 1 - \1_{\set{r^2 = \sigma^2}}$ and $\kappa_4' = \alpha/\sigma^4 -2 - \1_{\set{r^2 = \sigma^2}}$.
We will show that $s^2 - (1 + \1_{\set{r^2 = \sigma^2}})\sigma^2$ and $\alpha - (2+\1_{\set{r^2 = \sigma^2}})\sigma^4$ determine the infinitesimal law of the Wigner matrices. In a subsequent paper, \cite{mvb-r}, it will be shown that Eq. (\ref{eq:cauchy and r transform}) is an universal law giving the relation between the infinitesimal moments and infinitesimal cumulants of an infinitesimal law. This simplification shows that the asymptotic behavior of the matrices is governed by three quantities $\kappa_2$, $\kappa_2'$ and $\kappa_4'$. Moreover we extend their results in three ways. First, we shall consider independent families of Wigner matrices; second, we consider the joint distribution of the matrices and their transpose; thirdly, we make the class of Wigner matrices more general.  In this paper we shall refer to the $1/N$ corrections as the infinitesimal law of the matrix.

The particular case of Gaussian ensembles has been studied by Dumitriu and Edelman \cite{de}, Ledoux \cite{l}, and the second author \cite{m}. The joint infinitesimal law of a GUE matrix and its transpose was given in \cite{pst}, see Remark \ref{rem:pst connection}. Very recently, Muñoz George and Tseng \cite{mgt}, analyzed the infinitesimal law of Wishart matrices with non-Gaussian entries.

One of the main results of this paper is to demonstrate the asymptotic freeness of a family of $N \times N$ random matrices at the infinitesimal level, which is a significant strengthening of what was previously known. We do this by showing that the joint distribution of the family is (asymptotically) the joint distribution of a free family. We know that the mixed cumulants of a free family vanish. This means that if we can write each mixed moment as a sum over $\NC(n)$, then by Möbius inversion,  the terms must be the free cumulants. In addition if the terms can be factored such that each factor only involves one member of the family, then this shows that mixed cumulants do indeed vanish, and thus the joint distribution is the distribution of free random variables; this then shows that the family is free. 

In this paper we are concerned with the $1/N$ correction or the joint infinitesimal law of a family of random variables in a non-commutative probability space. In \cite{bs} Belinschi and Shlyakhtenko and in  \cite{fn} Février and Nica introduced a new kind of independence, which they called independence of type $B$, after the type $B$ cumulants of Biane, Goodman, and Nica \cite{bgn}. Later Shlyakhtenko \cite{s}, and Collins, Hasebe and Sakuma  \cite{chs} showed that the $1/N$ corrections of some ensembles followed the type $B$ independence rules, which we call here \textit{complex infinitesimal freeness}. The same equivalence between freeness and vanishing of mixed cumulants applies to complex  infinitesimal freeness because one can reduce complex infinitesimal freeness to operator valued freeness over a two dimensional commutative algebra (see Tseng \cite[Thm. 3.2]{t}). 

At this point, we need to introduce a weaker version of infinitesimal freeness presented by Février and Nica, which is called \textit{real infinitesimal freeness}. This new rule was necessary to accommodate the cases of random matrix ensembles which are not unitarily invariant, in particular the Wigner matrices discussed here. In \cite{m}, the second author showed that independent GOE random matrices are not asymptotically complex infinitesimal free; however in \cite{m} it was also shown that there is a universal rule for computing mixed moments. The combinatorial basis for this rule was elucidated using real Wishart matrices in \cite{mvb}. In \cite{cm}, Cébron and the second author found the cumulant form of this universal rule and called it  \textit{real infinitesimal freeness}. It was shown that independent and orthogonally invariant matrices are asymptotically real infinitesimally free. In this paper, we extend this asymptotic freeness property to Wigner matrices which only have invariance under the symmetric group; under an additional assumption about the covariance of a matrix entry and its complex conjugate we get the complex infinitesimal freeness of \cite{cm}. Asymptotic complex infinitesimal freeness was already demonstrated by Shlyakhtenko \cite{s}, and Collins, Hasebe and Sakuma \cite{chs} for unitarily invariant ensembles and fixed finite rank operators, under different terminology. 

A completely novel feature of real infinitesimal freeness is the appearance of the transpose of a matrix. What we see here is that the infinitesimal law of a single matrix will depend on the mixed moments of the matrix and its transpose. Fortunately for Wigner matrices, the dependence is only up to moments of order $4$, which we find convenient to express in terms of classical cumulants of the entries of the matrix. In \cite{mp}, Popa and the second author showed that a unitarily invariant matrix is asymptotically free from its transpose. We shall see the same here with a certain type of Wigner matrix. 

Going back to 't Hooft's famous discovery of planarity as the selection criterion  for highest order contributions, see \cite{g}, most authors have only considered the case of orientable surfaces. The appearance of the transpose has important topological consequences for random matrix theory as it brings non-orientable surfaces into the calculation.

In \cite{cm} it was shown that real infinitesimal freeness is equivalent to the vanishing of mixed real infinitesimal free cumulants. So if one can show that arbitrary mixed moments in a matrix and its transpose can be written as a sum of free cumulants summed over non-crossing partitions with no mixed cumulants, then one has shown that the joint distribution is that of freely independent random variables, and thus the variables are freely independent. This is exactly what we shall do.  

A simple fact, of which we shall make use, is that all the properties of infinitesimal freeness (real or complex) can be written in terms of upper triangular matrices, see \cite{t}.

So if we show that asymptotically arbitrary mixed moments can be written as a sum over $\NC(n)$ and each term factors into a product where each factor only contains one element of the family, then the factors are the free cumulants and we have demonstrated free independence.

In a  final section of the paper we give a bijection between walks on non-orientable graphs and non-crossing pairings on a non-orientable disc. While the bijection between non-crossing pairings and walks on trees goes back at least as far as Wigner's work, here we need to extend this to the non-orientable case, thus bringing the calculation into the realm of infinitesimal free probability and the infinitesimal $r$-transform of \cite{mvb-r}.


\subsection{Wigner matrices}

There is no standard definition of a Wigner matrix in the case of complex entries; the difference is mainly in what is assumed about the expectation of the square of an off-diagonal entry. Here are the assumptions we shall use in this paper. If one does not consider the $1/N$ correction or transposes, then the distinctions do not matter in the large $N$ limit.

\begin{definition}[Wigner Matrix]\label{def:wigner matrix}
    An $N \times N$ \emph{Wigner matrix} is a random matrix $W = \frac{1}{\sqrt{N}}(W_{ij})_{i,j}$, such that $(W_{ij})_{i \leq j}$ are complex independent random variables with diagonal entries $(W_{ii})_{i}$ being identically distributed, off-diagonal entries $(W_{ij})_{i < j}$ also being identically distributed and $W_{ij} = \overline{W_{ji}}$ for all $i, j$. We assume that diagonal and off-diagonal entries have moments of all orders; however only the first four are needed to determine our limiting distributions. Furthermore, we assume that $\esp\bracket*{W_{ij}} = 0$ for all $i, j$,
    \[
        \esp\bracket{W_{ij}^2} = r^2 \text{ if } i \neq j, \quad \esp\bracket{\abs{W_{ij}}^2} = \begin{cases}
          \sigma^2 & \text{if } i \neq j,\\
          s^2 & \text{if } i = j
        \end{cases}, \quad \esp\bracket*{\abs{W_{ij}}^4} = \alpha \text{ if } i \neq j,
    \]
    and
    \[
       \re\paren{\esp\bracket{W_{ij}^4}} = \beta \text{ if } i \neq j, \qquad \re\paren{\esp\bracket{W_{ij}^3\overline{W_{ij}}}} = \rho \text{ if } i \neq j,
    \]
    for some constants $r^2, \beta, \rho \in \R$ and $\sigma, s, \alpha > 0$. We assume that the joint distribution of the entries does not depend on $N$. 
\end{definition}

\begin{remark}\label{rmk:wigner matrix}
$(i)$\ Note that while $r^2$ can, in principle, take any complex value such that $\abs{r^2} \leq \sigma^2$, we restrict $r^2$ to the real numbers to simplify our calculations. Moreover, the entries of the ensemble $W$ are all real if and only if $r^2 = \sigma^2$. Dykema, \cite{kd}, showed that $r^2$ played no role in the limiting eigenvalue distribution of a family of independent Wigner matrices; Enriquez and Ménard \cite{em} showed that it does play a role in the infinitesimal law, although they only allowed $r^2 = 0$ or $r^2 = \sigma^2$ and considered Wigner matrices of the form $\frac{1}{\sigma}W$.

\smallskip\noindent$(ii)$\ For a GUE matrix $r^2 = 0$ and for a GOE matrix $r^2 = \sigma^2$. In \cite[\S 2.2]{agz}, it is assumed that $r^2 = 0$ and asymptotic freeness is proved in \cite[Theorem 5.4.2]{agz} under these assumptions. In \cite{kd}, asymptotic freeness of independent Wigner matrices is proved in \cite[Theorem 2.1]{kd} under assumptions weaker than both those in \cite{agz} and those made here, but transposes and infinitesimal laws are not considered. 



\end{remark}

\subsection{Statement of results} 

\begin{notation}
Given a random $N \times N$ matrix $W_j$, we will denote the entries of $W_j$ by $W_{kl}^{(j)}$. If $\epsilon \in \{-1, 1\}$, we let $W_j^{(\epsilon)}$ be $W_j$ if $\epsilon = 1$ and the transpose $W_j^t$ if $\epsilon = -1$. The $(k, l)$-entry of $W_j^{(\epsilon)}$ is $W_{kl}^{(j, \epsilon)}$. Since we have assumed that the joint distribution of entries is independent of $N$, we will not put $N$ into the notation.

\end{notation}

\begin{theorem}\label{thm:main}
Suppose $\{W_j\}_{j \in J}$ is an independent family of $N \times N$ Wigner matrices (as in Def. \ref{def:wigner matrix}). Then, the set $\{ W_j\}_{j \in J}$ is asymptotically semi-circular and real infinitesimally free. When $r^2 = 0$, the set $\{ W_j, W_j^t \}_{j \in J}$ is asymptotically free. Moreover, the real infinitesimal free cumulants of the limiting real semi-circular family are given by

\begin{itemize}

\item
$\kappa_2( w_{j_u}^{(\epsilon_u)}, w_{j_v}^{(\epsilon_v)} ) = k_2(W_{12}^{(j_u, \epsilon_u)}, W_{12}^{(j_v, -\epsilon_v)})$,

\item
\begin{align*}\lefteqn{
\kappa_2'( w_{j_u}^{(\epsilon_u)}, w_{j_v}^{(\epsilon_v)} ) } \\ 
& =
\esp\Big(W_{11}^{(j_u)} W_{11}^{(j_v)}\,\Big)
- 
\esp\Big(W_{12}^{(j_u, \epsilon_u)} \overline{W_{12}^{(j_v, \epsilon_v)}}\,\Big)
- 
\esp\Big(W_{12}^{(j_u, \epsilon_u)} \overline{W_{12}^{(j_v, -\epsilon_v)}}\,\Big),
\end{align*}

\item
\[
\kappa_4'(w_{j_t}^{(\epsilon_t)},  w_{j_u}^{(\epsilon_u)}, w_{j_v}^{(\epsilon_v)}, w_{j_y}^{(\epsilon_y)}) 
= \kern -1pt
\re\Big[k_4(W^{(j_t, \epsilon_t)}_{12}, \overline{W^{(j_u, \epsilon_u)}_{12}}, W^{(j_v, \epsilon_v)}_{12}, \overline{W^{(j_y, \epsilon_y)}_{12}})\Big],
\]
and

\item
all other infinitesimal cumulants vanish.

\end{itemize}

\end{theorem}

\begin{remark}
For an individual operator $w_j$, we have that the real free infinitesimal cumulants of $w_j$ are given by
\begin{itemize}

\item
$\kappa_2(w, w) = \esp(|W_{12}|^2)$,

\item
$\kappa_2(w, w^t) = \esp(W_{12}^2)$,

\item
$\kappa_2'(w, w) = \kappa_2'(w, w^t) = \esp(W_{11}^2 - |W_{12}|^2 - W_{12}^2)$,

\item
$\kappa_4'(w, w, w, w) = \esp(|W_{12}|^4) - 2 \esp(|W_{12}|^2)^2 -  \esp(W_{12}^2)^2 $.

\end{itemize}

\end{remark}

\subsection{Outline of the paper}
In \textbf{\S2} we present a review of infinitesimal probability spaces and the technique of counting walks on graphs that we shall use. In \textbf{\S3} we prove the main technical lemmas that will be used in the paper. In \textbf{\S4} we prove Theorem \ref{thm:first order freeness} which gives the asymptotic freeness of independent Wigner matrices and their transposes when $r^2 = 0$; this is the first part of the claim made in Theorem \ref{thm:main}. Note that the asymptotic freeness of independent Wigner matrices was proved by Dykema \cite{kd}; what is new here is the free independence of transposes. In addition we prove Lemma \ref{lemma:non-crossing part}, which shows that the contribution of the orientable graphs is described by the infinitesimal free cumulants $\kappa_2'$ and $\kappa_4'$. In \textbf{\S5} we prove Lemma \ref{lemma:symmetric annular case} which gives our bijection between the non-crossing annular pairings and the non-orientable graphs $\U_{i,k}(n)$. With these two Lemmas we conclude the proof of Theorem \ref{thm:main} and in \textbf{\S6} make some concluding remarks.

\section{Preliminaries}

This section presents the basic notions of non-commutative probability and real infinitesimal freeness (introduced by Cébron and the second author \cite{cm}) that will be used throughout the paper.

\subsection{Non-commutative probability spaces}

Let us recall from \cite{vdn} the notion of a non-commutative probability space.
\begin{definition}
We call the pair $(\cA, \phi)$ a \textit{non-commutative probability space} if
\begin{itemize}
\item
 $\cA$ is a unital algebra over $\C$; and
 
\item $\phi: \cA \to \C$ is linear with $\phi(1) = 1$;
\end{itemize}

We call the triple $(\cA, \phi, t)$ a \textit{real non-commutative probability space} if the following conditions are satisfied:
\begin{itemize}
\item
 $(\cA, \phi)$ is a  non-commutative probability space

\item
$\phi$ is \textit{tracial}: $\phi(ab) = \phi(ba)$;

\item
   $a \to a^t$ is an  anti-isomorphism of $\cA$: for $a, b \in \cA$
   
   \begin{itemize}

   \item
   $(a + b)^t = a^t + b^t$,

   \item
   $(\beta a)^t = \beta a^t$ (for $\beta \in \C$),

   \item
   $(a b)^t = b^t a^t$,
   
   \item
   $\phi(a^t) = \phi(a)$.
   
   \end{itemize}
\end{itemize}
\end{definition}
Recall that subalgebras $\cA_1, \dots, \cA_s \subseteq \cA$ in a non-commutative probability space $(\cA, \phi)$ are free if whenever we have  elements $a_1, \dots, a_n \in \cA$ with $\phi(a_i) = 0$ for $1 \leq i \leq n$ (i.e. the elements are \textit{centred}) and the tuple is \textit{alternating}: $a_i \in \cA_{j_i}$ with $j_1 \not = j_2 \not = \cdots \not = j_n$, we then have $\phi(a_1 \cdots a_n) = 0$. In this paper we shall make heavy use of the free cumulants of Speicher (see \cite{ns}). 

Let $\pi \in \NC(n)$ be a non-crossing partition. We recall that for a block $V = (i_1, \dots, i_l)$ we define  $\kappa_{|V|}(a_1, \dots, a_n | V) = \kappa_l(a_{i_1}, \dots, a_{i_l})$ and if the blocks of $\pi$ are $\{V_1, \dots, V_k\}$ and $a_1, \dots, a_n \in \cA$, we define
\[
\kappa_\pi(a_1, \dots, a_n) =
\prod_{i=1}^k \kappa_{|V_i|}(a_1, \dots, a_n | V_i).
\]
This produces the moment-cumulant formula for any $a_1, \dots, a_n \in \cA$
\[
\phi(a_1 \cdots a_n) =
\sum_{\pi \in \NC(n)} \kappa_\pi(a_1, \dots, a_n).
\]
By induction on $n$ this recursively defines the free cumulants $\{ \kappa_n \}_n$. Freeness is equivalent to the vanishing of mixed cumulants (see \cite[Thm.~11.16]{ns}).

\begin{remark}
Let $(\cA, \phi)$ be a non-commutative probability space. Recall that if  $w_1, \dots, w_s \in \cA$ are such that for $n \geq 3$, we have $\kappa_n(w_{i_1}, \dots, w_{i_n}) = 0$ for all $i_1, \dots, i_n \in [s]$, then we say that $\{w_1, \dots, w_s\}$ is a semi-circular family with covariance matrix $( \kappa_2(w_i, w_j) )_{i,j= 1}^s$. 
\end{remark}

\begin{definition}
Let  $w_1, \dots, w_s$ be elements of a real non-commutative probability space such that $\{w_1, w_1^t, \dots, w_s, w_s^t\}$ is a semi-circular family. Then we say that $\{w_1, \dots, w_s\}$ is a \textit{real semi-circular family}.
\end{definition}

\subsection{Infinitesimal probability space}

\begin{definition}
We call the tuple $(\cA, \varphi, \varphi', t)$ a \textit{real infinitesimal probability space} (a \textit{real infinitesimal non-commutative probability space}, to be precise), if the following conditions are satisfied:
\begin{itemize}
    \item $(\cA, \varphi, t)$ is a real probability space, and
    \item $\varphi' : \cA \to \C$ is linear with $\varphi'(1_\cA) = 0$.
\end{itemize}
\end{definition}

\subsection{Infinitesimal freeness}
Let us recall the basic definitions of infinitesimal freeness from \cite{fn}, here referred to as complex infinitesimal freeness. 

Let ($\cA, \phi, \phi')$ be an infinitesimal probability space and $\cA_1, \dots, \cA_s \subseteq \cA$ be unital subalgebras. We say that the subalgebras $\cA_1, \dots, \cA_s$ are \textit{infinitesimally free} if whenever we have $a_1, \dots, a_n \in \cA$ with
\begin{enumerate}

\item
$a_i \in \cA_{j_i}$ with $j_1 \not = j_2 \not = \cdots \not = j_n$, and

\item
$\phi(a_i) = 0$ for $1 \leq i \leq n$;

\end{enumerate}
we have $\phi(a_1 \cdots a_n) = 0$ and for $n = 2k - 1$ odd
\[
\phi'(a_1 \cdots a_n) = \phi(a_1 a_n) \phi(a_2 a_{n-1})
\cdots \phi(a_{k-1}a_{k+1}) \phi'(a_{k})
\]
and for $n$ even $\phi'(a_1 \cdots a_n) = 0$.

From the moment-cumulant formula we get the infinitesimal moment-cumulant formula by the Leibniz rule. First we set
\[
\partial\kappa_\pi(a_1, \dots, a_n) =
\sum_{i=1}^k \kappa'_{|V_i|}(a_1, \dots, a_n | V_i)
\prod_{j \not = i} \kappa_{|V_j|}(a_1, \dots, a_n | V_j).
\]
Then
\[
\phi'(a_1 \cdots a_n) = 
\sum_{\pi \in \NC(n)} \partial\kappa_\pi(a_1, \dots, a_n).
\]
Then infinitesimal freeness is equivalent to the vanishing of mixed cumulants (see \cite[Prop.~4.7]{fn}). By mixed cumulants we mean both Speicher's free cumulants and the infinitesimal free cumulants of Février and Nica. 



\subsection{Real infinitesimal freeness}

As mentioned in the introduction, working over the orthogonal group, or more generally with real Wigner matrices, requires a new kind of freeness, where the transpose map comes into play.
\begin{definition}
Let $(\cA, \varphi, \varphi', t)$ be a real infinitesimal probability space. Let, for each $i \in I$, $\cA_i \subset \cA$ be a symmetric unital sub-algebra. The sub-algebras $(\cA_i)_{i \in I}$ are called \emph{real infinitesimally free} if for each $n \in \N$, whenever $a_j \in \cA_{i_j}$ with $\varphi(a_j) = 0$ for all $1 \leq j \leq n$, and $i_1 \neq i_2, \ldots, i_{n-1} \neq i_n$, we have
\begin{itemize}
    \item $\varphi(a_1 \ldots a_n) = 0$; and
    \item when $n = 2$, $\varphi'(a_1a_2) = 0$,
    \item when $n = 2k - 1 \geq 3$, we have
    \begin{align*}
        \varphi'(a_1\cdots a_n) &= \varphi(a_1\varphi'(a_2 \ldots a_{n-1})a_n) +\\
        & \quad \varphi(a_1 a_k^t a_n)\varphi(a_2 a_{k+1}^t) \cdots \varphi(a_{k-1}a_{n-1}^t);
    \end{align*}
    \item when $n = 2k \geq 4$, we have
    \begin{align*}
        \varphi'(a_1\cdots a_n) &= \varphi(a_1\varphi'(a_2\cdots a_{n-1})a_n) +\\
        & \quad \varphi(a_1a_{k+1}^t)\varphi(a_2a_{k+1}^t)\cdots \varphi(a_k a_n^t).
    \end{align*}
\end{itemize}
\end{definition}


In a study of fluctuation moments of random matrices, \cite{mn}, Nica and the second author introduced \textit{non-crossing annular permutations}. The idea was to produce a combinatorial object that would play the role of non-crossing partitions in the analysis of fluctuation moments (see \cite[Ch.~5]{ms}). The properties needed for this paper are exposed in \cite{m, mvb,cm}, but we shall for convenience review them here. The definition of non-crossing here now requires us to work in  $S_n$, the symmetric group of permutations of $[n]$.

For $\pi \in S_n$ we let $\#(\pi)$ be the number of cycles in the cycle decomposition of $\pi$ (including cycles of length $1$). Let $\gamma_n \in S_n$ be the permutation with the single cycle $(1, 2, \dots, n)$. For a permutation $\pi$, its cycle decomposition produces a partition, which we also denote by $\pi$. The Euler formula for genus tells us that the partition $\pi$ is non-crossing if and only if the permutation $\pi$ satisfies $\#(\pi) + \#(\pi^{-1}\gamma_n) = n + 1$.

\setbox1=\hbox{\includegraphics[width=15em]{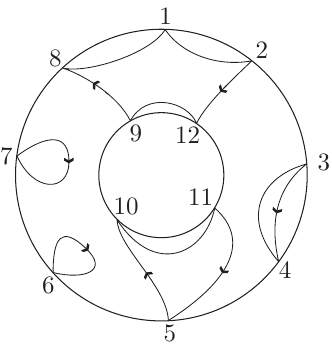}}
\setbox2=\hbox{\includegraphics[width=13.9em]{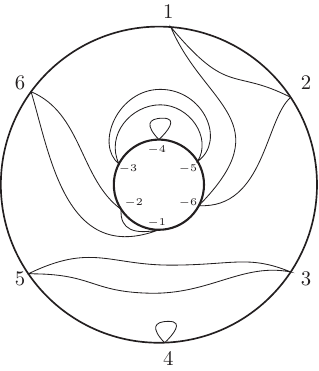}}

\begin{figure}
$\vcenter{\hsize=\wd1\box1}$ \hfill
$\vcenter{\hsize=\wd2\box2}$
\caption{\label{fig:annular-noncrossing-examples}\small On the left an element of $S_\NC(8,4)$ and on the right an element of $S_\NC^\delta(6, -6)$.}
\end{figure}

We now wish to extend this idea to the case where $m, n$ are positive integers and $\gamma_{m,n} = (1, 2, \dots, m)(m+1, \dots, m+n)$. We say that $\pi \in S_{m+n}$ is non-crossing annular if
\begin{itemize}

\item
$\#(\pi) + \#(\pi^{-1} \gamma_{m,n}) = m + n$, and

\item
$\pi \vee \gamma_{m, n} = 1_{m+n}$.

\end{itemize}
The second property states that the $\sup$ of the partition $\pi$ and the partition $\gamma_{m,n}$ has only one block; since $\gamma_{m,n}$ has two blocks, this property means that at least one cycle of $\pi$ meets both cycles of $\gamma_{m,n}$. Diagrammatically, this means that we can draw the cycles of $\pi$ in an $(m, n)$-annulus in such a way that the cycles don't cross and at least one cycle connects the two boundary circles. See Fig.~\ref{fig:annular-noncrossing-examples}. The set of non-crossing permutations of a $(m, n)$-annulus is denoted $S_\NC(m,n)$. In \cite{m,mvb} it was shown that a subset of $S_\NC(m,n)$, called $S_\NC^\delta(n, -n)$ was needed to handle the case of transposes. 

We let $[-n] = \{-1, -2, \dots, -n\}$ and $[\pm n]= [n] \cup [- n]$. Given a permutation $\pi \in S_n$ we extend $\pi$ to a permutation of $[\pm n]$, i.e. an element of $S_{\pm n}$, by setting $\pi(-k) = - k$ for $k \in [n]$. We then let $\delta \in S_{\pm n}$ be the permutation: $\delta(k) = -k$ for $k \in [ \pm n]$. Now we consider the permutations $\sigma \in S_{\pm n}$ such that
\begin{itemize}

\item 
$\sigma\delta$ has all cycles of length $2$,

\item
$\#(\sigma) + \#(\sigma^{-1}\gamma_n\delta\gamma_n^{-1}\delta) = 2n$; and

\item
$\sigma \vee \gamma_n\delta\gamma_n^{-1}\delta = 1_{[\pm n]}$.

\end{itemize}
In the third property above we mean the $\sup$ of the partition $\sigma$ and the partition $\gamma_n\delta\gamma_n^{-1}\delta$ is the partition of $[\pm n]$ with only one block. The first condition means that no cycle of $\sigma$ can contain a $k$ and $-k$, and in addition for every cycle $(i_1, \dots, i_k)$ of $\sigma$, $(-i_k, -i_{k-1}, \dots, -i_1)$ is also a cycle of $\sigma$. The second condition says that $\sigma$ is annular non-crossing. 

For $\sigma \in S_\NC^\delta(n, -n)$ and $a_1, \dots, a_n \in \cA$, we define $\kappa_{\sigma/2}(a_1, \dots, a_n)$ as follows. As mentioned above, every cycle of $\sigma$ occurs in a \textit{conjugate pair}: $c c'$, with $c' = \delta c^{-1} \delta$. By our assumptions about a real infinitesimal probability space and a cycle $c = (i_1, \dots, i_k)$, we have that $\phi(a_{i_1}^{(\epsilon_1)} \cdots a_{i_k}^{(\epsilon_k)}) = 
\phi(a_{i_k}^{(-\epsilon_k)} \cdots a_{i_1}^{(-\epsilon_1)})$
where $a^{(1)} = a$ and $a^{(-1)} = a^t$, as usual. This equality of moments then implies equality of cumulants so $\kappa_k(a_{i_1}^{(\epsilon_1)}, \dots, a_{i_k}^{(\epsilon_k)}) = 
\kappa_k(a_{i_k}^{(-\epsilon_k)}, \dots, a_{i_1}^{(-\epsilon_1)})$. This means that when we choose one representative from each conjugate pair, the result $\kappa_{\sigma/2}(a_1, \dots, a_n)$ is independent of the choice.

\begin{definition}
Let $(\cA, \varphi, \varphi', t)$ be a tracial real infinitesimal probability space. For $a_1, \ldots, a_n \in \cA$, we set $\kappa_1'(a_1) = \varphi'(a_1)$ for $n = 1$, and for $n \geq 2$
\begin{equation}\label{eq:real moment cumulant}
    \varphi'(a_1\cdots a_n) = \sum_{\pi \in \NC(n)} \partial\kappa_{\pi}(a_1, \ldots, a_n) + \sum_{\sigma \in S^{\delta}_{\NC}(n, -n)} \kappa_{\sigma/2}(a_1, \cdots, a_n).
\end{equation}
\end{definition}

In \cite[Thm.~7.1]{cm} it was shown that real infinitesimal freeness was equivalent to the vanishing of mixed real free infinitesimal cumulants. 

\section{Basic calculations: reduction to graphs }

Many of the techniques employed here were also used by Muñoz George and the second author in \cite{mm1, mm2}, in particular \S\ref{subsec:graph associated to a partition}. However, we won't assume anything from these papers. Let $\set{W_{j}}_{j \in J}$ be a family of independent identically distributed $N \times N$ Wigner random matrices. For $\epsilon \in\Z_2 = \{ -1, 1\}$, we set $W_j^{(\epsilon)} = W_j$ for $\epsilon = 1$, and $W_j^{(\epsilon)} = W_j^t$ for $\epsilon = - 1$. For $j \in J$ and $\epsilon \in\Z_2$, we denote by $W_{kl}^{(j, \epsilon)}$ the $(k,l)$-entry of 
 $W_k^{(\epsilon)}$. 

For $\bm \epsilon = (\epsilon_1, \dots, \epsilon_n) \in \Z_2^n$ and  $\bm j = (j_1, \ldots, j_n) \in J^n$, we compute the coefficients of order $N^0$ and $N^{-1}$ of the mixed moments:
\[
    \esp\bracket*{\tr\paren*{W_{j_1}^{(\epsilon_1)}  \cdots W_{j_n}^{(\epsilon_n)} }} = N^{-(n/2 + 1)} \kern-0.5em \sum_{\bm i : [n] \to [N]} \esp\bracket*{W_{i_1 i_2}^{(j_1, \epsilon_1)} W_{i_2 i_3}^{(j_2, \epsilon_2)} \cdots  W_{i_n i_1}^{(j_n, \epsilon_n)}},
\]
where $\tr$ is the normalized trace. We shall extend to $(\cA, \phi, \phi', t)$ the $\epsilon$-notation as follows. If $w \in \cA$ and $\epsilon \in \Z_2$, we set
\[
w^{(\epsilon)} =
\begin{cases} w & \epsilon = 1,\\ w^t & \epsilon = -1.\end{cases}
\]
Our first goal will be to show that $\esp\bracket*{\tr\paren*{W_{j_1}^{(\epsilon_1)}  \cdots W_{j_n}^{(\epsilon_n)} }} $ converges to $\phi(w_{j_1}^{(\epsilon_1)}\ab \cdots w_{j_n}^{(\epsilon_n)})$, where $\{w_1, \dots, w_s\}$ is a real semi-circular family with covariance matrix given by
\[
\kappa_2(w_i, w_i) = \esp [| W_{12}|^2] = \sigma^2 \mbox{\ and\ }
\kappa_2(w_i, w_i^t) = \esp [W_{12}^2] = r^2,
\]
and for $\epsilon \in \Z_2$, $\kappa_2(w_i, w_j^{(\epsilon)}) = 0$ whenever $i \not = j$. 

Our second goal is to show that
\[
N\paren*{\esp\bracket*{\tr\paren*{W_{j_1}^{(\epsilon_1)}  \cdots W_{j_n}^{(\epsilon_n)} }} - \phi\paren*{w_{j_1}^{(\epsilon_1)}\ab \cdots w_{j_n}^{(\epsilon_n)}}}
\]
converges to $\phi'(w_{j_1}^{(\epsilon_1)}\ab \cdots w_{j_n}^{(\epsilon_n)})$, with $\{w_1, \dots, w_s\}$ real infinitesimally free with infinitesimal law given by: for $\epsilon_1, \ldots, \epsilon_4 \in \Z_2$,
\[
\kappa_2'(w_i^{(\epsilon_1)}, w_i^{(\epsilon_2)}) =
\esp\Big(W_{11}^2
- W_{12}^{(\epsilon_1)} \overline{W_{12}^{(\epsilon_2)}}
- W_{12}^{(\epsilon_1)} \overline{W_{12}^{(-\epsilon_2)}}\,\Big),
\] 
\[
\kappa_4'(w_{i}^{(\epsilon_1)},  w_{i}^{(\epsilon_2)}, w_{i}^{(\epsilon_3)}, w_{i}^{(\epsilon_4)}) 
=
\re\Big(k_4\big(W^{(\epsilon_1)}_{12}, \overline{W^{(\epsilon_2)}_{12}}, W^{(\epsilon_3)}_{12}, \overline{W^{(\epsilon_4)}_{12}}\big)\Big),
\]
the mixed cumulants are $0$, and all other $\kappa_n'$'s are $0$. 

\subsection{The graph associated to a partition}\label{subsec:graph associated to a partition}

Consider the permutation $\gamma_n = (1, 2,\ldots, n)$. When there is no risk of confusion, we will write $\gamma$ instead of $\gamma_n$. Let $G_n$ be the oriented labelled graph with set of vertices $[n]$ and set of edges
\[
    \set{(\gamma(k), k) : k \in [n]},
\]
where the \textit{label} of the edge $(\gamma(k), k)$ is $k$. 

For a partition $\pi \in \PP(n)$, we define the oriented labelled multi-graph $G(\pi)$ by the quotient of $G_n$ with respect to $\pi$, which is obtained by identifying the vertices of $G_n$ that belong to the same block in $\pi$. Then $G(\pi)$ is a graph with set of vertices $\pi$ and set of edges
\[
    \set*{e_k := ([\gamma(k)]_\pi, [k]_\pi) : k \in [n]},
\]
where $[k]_\pi$ denotes the block of $\pi$ that contains $k$ and the \textit{label} of the edge $e_k$ is $k$. In addition, the permutation $\gamma^{-1}$, read as a cycle on the vertices of $G(\pi)$, is an Eulerian circuit\footnote{An Eulerian circuit on an oriented graph is a closed walk that visits each directed edge exactly once, starting and ending at the same vertex. To exist, every vertex of the graph must have an equal number of incoming and outgoing edges.} on the graph. Therefore, the graph $G(\pi)$ is connected (considering the orientation of the edges).

The graph $G(\pi)$ induces a new partition $\overline\pi \in \PP(n)$ whose blocks consist of the labels of the edges in $G(\pi)$ that have the same set of the endpoints (regardless of orientation).
We let  $\overline{G}(\pi)$  be the elementary graph  obtained by identifying the edges of $G(\pi)$ that belong to the same block in $\overline\pi$ and ignoring the orientation of the edges. Then the graph $\overline G(\pi)$ has $\#(\pi)$ vertices given by the blocks in $\pi$ and $\#(\overline\pi)$ edges given by the set
\[
    \set*{\set{[\gamma(k)]_\pi, [k]_\pi} : k \in [n]}.
\]
Since $G(\pi)$ is connected, the graph $\overline G(\pi)$ is connected. Therefore, the following condition is satisfied
\begin{equation}\label{eq:vertices_edges_relation}
    \#(\pi) \leq \#(\overline\pi) + 1.
\end{equation}

Let $\bm j \in J^n$ and $\bm \epsilon \in \Z_2^n$ be fixed and let $\bm i : [n] \to [N]$. Denote by $\ker(\bm i) \in \PP(n)$ the partition of $[n]$ such that $u \sim v$ if $i_u = i_v$. Since $\set{(W_{kk}^{(j)})_k, (W_{kl}^{(j)})_{k < l}}_{j \in J}$ are independent families of i.i.d. random variables, we have that
\[
    \esp\bracket*{W_{i_1i_2}^{(j_1, \epsilon_1)} \cdots W_{i_n i_1}^{(j_n, \epsilon_n)}} = \esp\bracket*{W_{i'_1 i'_2}^{(j_1, \epsilon_1)} \cdots W_{i'_n i'_1}^{(j_n, \epsilon_n)}}
\]
whenever $\ker(\bm i) = \ker(\bm i')$ and an order condition is satisfied, see below. 

Indeed, let $\pi = \ker(\bm i) = \{V_1, \dots, V_k\}$ and $v_1, \dots, v_k \in[N]$ be $k$ distinct values and set $\bm i|_{V_l} = v_l$. Then
\begin{equation*}
    \esp\bracket*{W_{i_1i_2}^{(j_1, \epsilon_1)} \cdots W_{i_n i_1}^{(j_n, \epsilon_n)}}
\end{equation*}
depends only on the order of the set $\{v_1, \dots, v_k\}$. We set
\begin{equation}\label{eq:def_E_pi_j_eps}
    E(\pi, \bm j, \bm \epsilon) = \frac{1}{k!} 
    \sum_{\bm i|_{V_l} = v_{\sigma(l)}}
    \esp\bracket*{W_{i_1i_2}^{(j_1, \epsilon_1)} \cdots W_{i_n i_1}^{(j_n, \epsilon_n)}},
\end{equation}
where we sum $\sigma$ over all $k!$ permutations of $\{v_1, \dots, v_k\}$. Then, since we can choose $\set{v_1, \ldots, v_k}$ from $[N]$ in $\binom{N}{k}$ different ways,
\begin{equation}\label{eq:sum_E_pi_j_eps}
    \sum_{\ker(\bm i) = \pi} \esp\bracket*{W_{i_1i_2}^{(j_1, \epsilon_1)} \cdots W_{i_n i_1}^{(j_n, \epsilon_n)}}
= (N)_{\#(\pi)} E(\pi,\bm j, \bm \epsilon),
\end{equation}
where
$$(N)_{\#(\pi)} = N(N - 1) \cdots (N - \#(\pi) + 1).$$




\noindent
Note that the expectation factorizes according to the blocks of $\overline{\pi}$ as
\begin{equation}\label{eq:E_pi_j_eps_factorization}
    \esp\bracket*{W_{i_1i_2}^{(j_1, \epsilon_1)} \cdots W_{i_n i_1}^{(j_n, \epsilon_n)}} = \prod_{\set{l_1, \ldots,  l_k} \in \overline\pi} \esp\bracket*{W_{i_{l_1}i_{\gamma(l_1)}}^{(j_{l_1}, \epsilon_{l_1})}\cdots W_{i_{l_k}i_{\gamma(l_k)}}^{(j_{l_k}, \epsilon_{l_k})}}.
\end{equation}

\begin{remark}\label{rmk:block_size_two}
Let $\bm i : [n] \to [N]$ and set $\pi = \ker(\bm i)$. If $\#(\pi) = 1$, then
\[
    E(\pi, \bm j, \bm \epsilon) = \esp\bracket*{W_{i_1i_1}^{(j_1, \epsilon_1)} \cdots W_{i_1i_1}^{(j_n, \epsilon_n)}}.
\]
If in addition $n$ is even, $\#(\pi) = 2$, $\#(\overline\pi) = 1$ and the number of edges of $G(\pi)$ in one orientation is the same as the number of edges in the opposite orientation, then
\begin{align*}
    E(\pi, \bm j, \bm\epsilon)
    &= \frac{1}{2}\paren*{\esp\bracket*{W_{i_1i_2}^{(j_1, \epsilon_1)} \cdots W_{i_2i_1}^{(j_n, \epsilon_n)}} + \esp\bracket*{W_{i_1i_2}^{(j_1, -\epsilon_1)} \cdots W_{i_2i_1}^{(j_n, -\epsilon_n)}}}\\
    &= \re\paren*{\esp\bracket*{W_{i_1i_2}^{(j_1, \epsilon_1)} \cdots W_{i_2i_1}^{(j_n, \epsilon_n)}}}.
\end{align*}
In particular, since $r^2$ and $\sigma^2$ are real numbers, when $n = 2$ it follows that
\[
    E(\pi, \bm j, \bm \epsilon) = \esp\bracket*{W_{i_1i_2}^{(j_1, \epsilon_1)} W_{i_2i_1}^{(j_2, \epsilon_2)}}.
\]
\end{remark}

Let $S = \set{s_1 < \ldots < s_n}$ be a finite totally ordered set. We extend the definition of $G(\pi)$, $\overline G(\pi)$, $\overline\pi$ and $E(\pi, \bm j, \bm \epsilon)$ for $\pi \in \PP(S)$ by identifying $S$ with $[n]$ via the unique order-preserving isomorphism given by $s_k \mapsto k$. Under this identification, the structure of $G(\pi)$ is determined by the permutation $\gamma$, while the edges remain explicitly indexed by the elements of $S$. Specifically, we have $e_{s_k} = ([s_{\gamma(k)}]_\pi, [s_k]_\pi)$.

\begin{lemma}\label{lemma:edge upper bound}
    If $\overline\pi$ has a singleton, then $E(\pi, \bm j, \bm \epsilon) = 0$. In particular, if $\pi$ is such that $\#(\pi) > n/2 + 1$, then $E(\pi, \bm j, \bm \epsilon) = 0$.
    
\end{lemma}

\begin{proof}
    If $\overline\pi$ has a singleton, by \eqref{eq:def_E_pi_j_eps} and \eqref{eq:E_pi_j_eps_factorization}, we get that $E(\pi, \bm j, \bm \epsilon) = 0$. In particular, if $\#(\pi) > n/2 + 1$, it follows by \eqref{eq:vertices_edges_relation} that $\#(\overline\pi) > n/2$. Therefore, the partition $\overline{\pi}$ has a singleton. Hence, $E(\pi, \bm j, \bm \epsilon) = 0$.    
\end{proof}

\begin{remark}
In consequence, if $E(\pi, \bm j, \bm \epsilon) \neq 0$ then each block of $\overline\pi$ is of size at least two. Thus, 
\begin{equation}\label{eq:edge upper bound}
    \#(\pi) - 1 \leq \#(\overline{\pi}) \leq n/2.
\end{equation}
\end{remark}


Now, by \eqref{eq:sum_E_pi_j_eps}, we can write
\[
    \esp\bracket*{\tr\paren*{
    W_{j_1}^{(\epsilon_1)} \cdots W_{j_n}^{(\epsilon_n)}}} 
    = 
    \sum_{\pi \in \PP(n)} (N)_{\#(\pi)}N^{-(n/2 + 1)} E(\pi, \bm j, \bm \epsilon).
\]
Moreover,
\[
    \frac{(N)_{\#(\pi)}}{N^{n/2 + 1}} = N^{\#(\pi) - \frac{n}{2} - 1} - \binom{\#(\pi)}{2} N^{\#(\pi) - \frac{n}{2} - 2} + O(N^{\#(\pi) - \frac{n}{2} -3}).
\]
Then, the only cases where a term of order $N^0$ or a term of order $N^{-1}$ appears in the expansion of the mixed moment are: $\#(\pi) = n/2 + 1$ and $\#(\pi) = n/2$.

Consequently,
\begin{equation}\label{eq:expansion_mixed_moment}
    \esp\bracket*{\tr\paren*{W_{j_1}^{(\epsilon_1)} \cdots W_{j_n}^{(\epsilon_n)}}} = A_n + N^{-1} \paren*{B_n - \binom{\frac{n}{2} + 1}{2}A_n} + O(N^{-2}),
\end{equation}
where
\begin{align*}
    A_n &= \sum_{\substack{\pi \in \PP(n)\\ \#(\pi) = n/2 + 1}} \kern-0.75em  E(\pi, \bm j, \bm \epsilon), 
    \qquad B_n = \sum_{\substack{\pi \in \PP(n)\\ \#(\pi) = n/2}}\kern-0.5em  E(\pi, \bm j, \bm \epsilon).
\end{align*}
Observe that $A_n = B_n = 0$ for $n$ odd. Assume from now on that $n$ is even.

\section{The Asymptotic Expansion of Mixed Moments}

\subsection{The leading term $A_n$}

\begin{definition}
We say that $G(\pi)$ is a \textit{double-tree} if $\overline{G}(\pi)$ is a tree and $\overline\pi$ is a pairing such that paired edges have opposite orientations. 
\end{definition}
We prove, subject to certain constraints, that the only partitions $\pi \in \PP(n)$ with $\#(\pi) = \frac{n}{2} + 1$ and $E(\pi, \bm j, \bm \epsilon) \neq 0$ are those whose associated graph $G(\pi)$ is a double-tree.

\begin{definition}\label{def:turn-around}
We define a \textit{turn-around point} of $\pi$ to be an element $k$ in a block $V \in \pi$ such that $\gamma^{-1}(k) \sim_{\pi} \gamma(k)$ and $k \not\sim_{\pi} \gamma(k)$. More succinctly, we say that $k$ is a turn-around point of the block $V$. In addition, if $\gamma^{-1}(k) \sim_{\pi} \gamma(k)$ then $\gamma^{-1}(k) \sim_{\overline\pi} k$. If $V$ is a singleton with a turn-around point, we call $V$ a \textit{double-leaf} of $\pi$. Also, given $A \subseteq [n]$, we define
\[
    \pi|_A = \set{V \cap A : V \in \pi, V \cap A \neq \varnothing}.
\]
\end{definition}

\begin{lemma}\label{lemma:pruning}
Let $\pi \in \PP(n)$ with a singleton $\set{k}$. The block $\set{k}$ is a double-leaf of $\pi$ if and only if $\set{\gamma^{-1}(k), k} \in \overline{\pi}$. Suppose $\set{k}$ is a double-leaf of $\pi$ and let $\rho = \pi|_{[n] \setminus \set{\gamma^{-1}(k), k}}$.  
Then
\begin{itemize}
    \item $\overline{\pi} = \set{\gamma^{-1}(k), k} \cup \overline{\rho}$,
    \item $\overline{\pi} \in \NC(n)$ if and only if $\overline{\rho} \in \NC([n]\setminus \{ \gamma^{-1}(k), k\})$.
\end{itemize}

\end{lemma}

\begin{proof}
Note that $\gamma^{-1}(k)$ and $k$ are the only edges in $G(\pi)$ with $\set{k}$ as an endpoint. Thus, $\gamma^{-1}(k) \sim_{\pi} \gamma(k)$ if and only if $\gamma^{-1}(k) \sim_{\overline\pi} k$. Moreover, since $\set{k}$ is a singleton of $\pi$, then $\gamma^{-1}(k) \sim_{\overline\pi} k$ if and only if $\set{\gamma^{-1}(k), k} \in \overline{\pi}$. This proves the first claim.
    
Now assume that $\set{k}$ is a double-leaf. Then
\[
    \overline{\pi} = \set{\gamma^{-1}(k), k} \cup \overline{\rho},
\]
where $\rho = \pi|_{[n] \setminus \set{\gamma^{-1}(k), k}}$. In addition, since $\set{\gamma^{-1}(k), k}$ is a cyclic interval in $\overline{\pi}$ and $\set{k}$ is a singleton in $\pi$, we conclude that $\overline{\pi}$ is non-crossing if and only if $\overline{\rho}$ is non-crossing.
\end{proof}

\begin{remark}\label{rmk:pruning}
Let $\pi \in \PP(n)$ with a double-leaf $\set{k}$. Let $\pi_1 = \pi|_{[n] \setminus \set{\gamma^{-1}(k), k}}$ and  $\pi_2 = \set{\set{\gamma^{-1}(k)}, \set{k}}$. By the construction of $G(\pi)$ and Lemma \ref{lemma:pruning} above, the graph $G(\pi_1)$ is obtained from $G(\pi)$ by deleting the vertex $\set{k}$ (and  the edges $\gamma^{-1}(k)$, $k$) and then identifying $\gamma^{-1}(k)$ with $\gamma(k)$ in $\pi$. 

Let $\bm j_1 = \bm j|_{[n] \setminus \set{\gamma^{-1}(k), k}}$ and $\bm j_2 = \bm j|_{\set{\gamma^{-1}(k), k}}$. Note that the deleted edges have opposite orientations. Let $\bm \epsilon_1 = \bm \epsilon|_{[n] \setminus \{ \gamma^{-1}(k), k\}}$ and $\bm \epsilon_2 = \bm \epsilon|_{\{ \gamma^{-1}(k), k\}}$.  Moreover, since $\set{k} \in \pi$, $\set{\gamma^{-1}(k), k} \in \overline\pi$, $\gamma^{-1}(k) \sim_\pi \gamma(k)$, using Remark~\ref{rmk:block_size_two} it follows that
\[
    E(\pi, \bm j, \bm \epsilon) = E(\pi_1, \bm j_1, \bm \epsilon_1)
    E(\pi_2, \bm j_2, \bm \epsilon_2).
\]     
\end{remark}
\begin{notation}
We denote by $\DT(S) \subseteq \PP(S)$ the set of partitions $\pi$ whose associated graph $G(\pi)$ is a double-tree. For $\pi \in \DT(n)$, let $f(\pi, \bm \epsilon) = r^{2p} \sigma^{2 q}$ where $p$ is the number of pairs $(u, v) \in \overline\pi$ with $\epsilon_u = - \epsilon_v$ and $q$ is the number of pairs $(u, v) \in \overline\pi$ with $\epsilon_u =  \epsilon_v$.
\end{notation}

\begin{lemma}\label{lemma:double-trees}
Let $\pi \in \PP(n)$ with $\#(\pi) = n/2 + 1$. Then, for $r^2 \not = 0$, $E(\pi, \bm j, \bm \epsilon) \neq 0$ if and only if $\pi \in \DT(n)$ and $\overline\pi \leq \ker(\bm j)$. For $r^2 = 0$, the condition becomes $\pi \in \DT(n)$ and $\overline\pi \leq \ker(\bm j) \wedge \ker(\bm \epsilon)$. When either of these conditions holds, we have that $\overline\pi \in \NC_2(n)$, $\Kr(\pi) = \overline\pi$ and $E(\pi, \bm j, \bm \epsilon) = f(\pi , \bm \epsilon)$.
\end{lemma}

\begin{proof}
We proceed by induction on $n$. If $n = 2$, the only partition $\pi \in \PP(n)$ with $\#(\pi) = \frac{n}{2} + 1$ is $\pi = \set{\set{1}, \set{2}}$. By the construction of $G(\pi)$, we get that $\pi \in \DT(2)$ and $\overline\pi = \set{\set{1,2}}$. Therefore, $\overline\pi \in \NC_2(n)$ and $\Kr(\pi) = \overline\pi$. Moreover,
\begin{equation}\label{eq:e equals two}
    E(\pi, \bm j, \bm \epsilon) = \begin{cases}
        \sigma^2 & \text{if } \ker(\bm j) = \ker(\bm \epsilon) = 1_2,\\
        r^2 & \text{if } \ker(\bm j)  = 1_2 \not = \ker(\bm \epsilon), \\
        0 & \text{otherwise.}
    \end{cases}
\end{equation}
When $r^2 \neq 0$, we then get $E(\pi, \bm j, \bm \epsilon) \neq 0$ if and only if $\overline\pi \leq \ker(\bm j)$. When $r^2 = 0$, 
the condition becomes $\overline\pi \leq \ker(\bm j) \wedge \ker(\bm \epsilon)$.

Suppose that the statement of the lemma is satisfied for some even integer $n \geq 2$. Let $\pi \in \PP(n+2)$ be a partition such that $\#(\pi) = \frac{n}{2} + 2$. When $r^2 \not = 0$, if $\pi \in \DT(n+2)$ and $\overline{\pi} \leq \ker(\bm j)$, it follows by \eqref{eq:E_pi_j_eps_factorization} and Remark~\ref{rmk:block_size_two} that 
$E({\pi}, \bm j, \bm \epsilon) = r^{2p} \sigma^{2q} \not= 0$.
When $r^2 = 0$, by the same argument as before, it follows that $E({\pi}, \bm j, \bm \epsilon) = \sigma^{n} \not= 0$ when $\overline{\pi} \leq \ker(\bm j) \wedge \ker(\bm \epsilon)$.
    
Assume that $E(\pi, \bm j, \bm \epsilon) \neq 0$. By Lemma \ref{lemma:edge upper bound} and \eqref{eq:edge upper bound}, it follows that $\overline\pi$ is a pairing. Note that $\overline G(\pi)$ is a connected graph with $\frac{n}{2} + 2$ vertices and $\frac{n}{2} + 1$ edges. Therefore $\overline G(\pi)$ is a tree and it must have a leaf. Let $V \in \pi$ be a leaf in $\overline G(\pi)$. Since $\overline\pi$ is a pairing, exactly two edges have $V$ as an endpoint in $G(\pi)$. By the construction of $G(\pi)$, we obtain that $V = \set{k}$ is a singleton such that $\set{\gamma^{-1}(k), k} \in \overline\pi$. Then, by Lemma \ref{lemma:pruning}, it follows that $\set{k}$ is a double-leaf and $\overline{\pi} = \overline\pi_1 \cup \overline\pi_2$, where
\[
    \pi_1 = \set{\set{\gamma^{-1}(k)}, \set{k}}, \qquad \pi_2 = \pi|_{[n+2] \setminus \set{\gamma^{-1}(k), k}}.
\]
    
Let $\bm j_1 = \bm j|_{\set{\gamma^{-1}(k), k}}$, $\bm \epsilon_1 = \bm \epsilon|_{\set{\gamma^{-1}(k), k}}$, $\bm j_2 = \bm j|_{[n+2] \setminus \set{\gamma^{-1}(k), k}}$ and $\bm \epsilon_2 = \bm \epsilon|_{[n+2] \setminus \set{\gamma^{-1}(k), k}}$. By Remark \ref{rmk:pruning}, it follows that
\[
    E(\pi, \bm j, \bm \epsilon) = E(\pi_1, \bm j_1, \bm \epsilon_1)E(\pi_2, \bm j_2, \bm \epsilon_2).
\]
Moreover, since $E(\pi, \bm j, \bm \epsilon) \neq 0$, we get that 
$$E(\pi_1, \bm j_1, \bm \epsilon_1) \neq 0 \qquad \text{and} \qquad E(\pi_2, \bm j_2, \bm \epsilon_2) \ab\neq 0.$$

Using the base case, we obtain that $\pi_1 \in \DT(\set{\gamma^{-1}(k), k})$ is a double-tree, $\Kr(\pi_1) = \overline\pi_1$, $\overline\pi_1 \leq \ker(\bm j_1)$ and
\[
E(\pi_1, \bm j_1, \bm \epsilon_1) = \begin{cases} \sigma^2 & \overline\pi_1 \leq \ker(\bm \epsilon_1), \\ r^2 & \overline\pi_1 \not\leq \ker(\bm \epsilon_1).
\end{cases}
\]
By the induction hypothesis, we get that $\pi_2 \in \DT([n+2] \setminus \set{\gamma^{-1}(k), k})$ is a double-tree and  that $\overline\pi_2 \leq \ker(\bm j_2)$. Moreover, $\overline\pi_2 \in \NC_2([n+2] \setminus \set{\gamma^{-1}(k), k})$, $\Kr(\pi_2) = \overline\pi_2$ and $E(\pi_2, \bm j_2, \bm \epsilon_2) = r^{2 p_2} \sigma^{2q_2}$, where $p_2$, respectively $q_2$, are the number of pairs $(u, v)$ in $\overline\pi_2$ where $\epsilon_u = - \epsilon_v$, respectively $\epsilon_u =  \epsilon_v$. Thus $E(\pi, \bm j, \bm \epsilon) = r^{2 p} \sigma^{2 q} = f(\pi, \bm \epsilon)$.

We conclude that $G(\pi)$ is a double-tree because $G(\pi)$ is obtained by joining the two double-trees $G(\pi_1)$ and $G(\pi_2)$  at a vertex. That is, $\pi \in \DT(n+2)$. In addition, $\overline\pi = \overline{\pi_1} \cup \overline{\pi_2} \leq \ker(\bm j)$. As $\overline\pi$ is obtained from $\overline\pi_2\in \NC_2([n+2] \setminus \set{\gamma^{-1}(k), k})$ by inserting the pair of cyclically adjacent elements in $\overline\pi_1$, we get that $\overline\pi \in \NC_2(n+2)$. 

Since $\{k\}$ is a singleton in $\pi$ and $\gamma^{-1}(k) \sim_\pi \gamma(k)$, we have that $\Kr(\pi_1) = \{\gamma^{-1}(k), \gamma(k)\}$ is a block of $\Kr(\pi)$; all the other blocks of $\Kr(\pi)$ are the blocks of $\Kr(\pi_2)$. Hence $\Kr(\pi)= \Kr(\pi_1) \cup \Kr(\pi_2)  = \overline\pi$. 

When $r^2 = 0$, we just have to impose the condition $\overline \pi \leq \ker(\bm \epsilon)$ to get $f(\pi, \bm \epsilon) = \sigma^n$.
\end{proof}

\begin{lemma}\label{lemma:description-double-trees}
Let $\pi \in \PP(n)$. Then, $\pi \in \DT(n)$ if and only if $\overline\pi = \Kr(\pi) \in \NC_2(n)$.
\end{lemma}

\begin{proof}
If $\pi \in \DT(n)$, by setting $\bm j = \bm \epsilon = (1, \ldots, 1)$ and $r^2 = \sigma^2$ in Lemma \ref{lemma:double-trees}, it follows that $\overline\pi = \Kr(\pi) \in \NC_2(n)$.

Assume now that $\pi \in \PP(n)$ with $\overline\pi = \Kr(\pi) \in \NC_2(n)$. We proceed by induction on $n$ to prove that $\pi \in \DT(n)$. If $n = 2$, then $\overline\pi = \set{\set{1,2}}$. Hence, $\pi = \set{\set{1},\set{2}} \in \DT(n)$.

Suppose that the statement is true for some $n \geq 2$. We prove that the statement is satisfied for $\overline\pi = \Kr(\pi) \in \NC_2(n + 2)$. Since $\overline\pi$ is a non-crossing pairing, we have that $\overline\pi$ contains an interval $\set{\gamma^{-1}(k), k} \in \overline\pi$. Therefore, since $\overline\pi = \Kr(\pi)$, it follows that $\set{k} \in \pi$. Then, by Lemma \ref{lemma:pruning}, $\set{k}$ is a double-leaf of $G(\pi)$ and $\overline\pi = \overline{\pi}_1 \cup \overline\pi_2$, where
\[
    \pi_1 = \set{\set{\gamma^{-1}(k)}, \set{k}}, \qquad \pi_2 = \pi|_{[n+2] \setminus \set{\gamma^{-1}(k), k}}.
\]

Note that $\pi_1 \in \DT(\set{\gamma^{-1}(k), k})$ and, since $\set{k} \in \pi$ is a singleton, the partition $\pi_2 \in \PP(n)$ is such that $\overline\pi_2 = \Kr(\pi_2) \in \NC_2([n+2] \setminus \set{\gamma^{-1}(k), k})$. Then, by induction hypothesis, it follows that $\pi_2 \in \DT([n+2] \setminus \set{\gamma^{-1}(k), k})$. Finally, since $G(\pi)$ is obtained from the two double-trees $G(\pi_1)$ and $G(\pi_2)$ by joining at a vertex, we conclude that $\pi \in \DT(n+2)$.
\end{proof}

Using Lemma \ref{lemma:double-trees}, by the fact that $\Kr^{-1} : \NC_2(n) \to \PP(n)$ is one-to-one, it follows that
\begin{align}\label{eq:a expression}
    A_n  &= \sum_{\pi \in \DT(n)} f(\pi, \bm \epsilon) 
    \1_{\overline\pi \leq \ker(\bm j)} = \sum_{\pi \in \NC_2(n)} f(\Kr^{-1}(\pi), \bm \epsilon) 
    \1_{\pi \leq \ker(\bm j)} \notag\\
    & =
    \sum_{\pi \in \NC(n)} \kappa_\pi(w_{j_1}^{(\epsilon_1)}, \dots, w_{j_n}^{(\epsilon_n)}),
\end{align}
where $\kappa_2(w_i, w_i) = \sigma^2$ and $\kappa_2(w_i, w_i^t) = r^2$, and all other cumulants are $0$. This proves the first part of Theorem \ref{thm:main}.

\begin{theorem}\label{thm:first order freeness}
The limiting joint distribution of $\{W_j, W_j^t\}_{j \in J}$ is the joint distribution of a semi-circular family $\{w_j, w_j^t\}_{j \in J}$. Let $\cA_j$ be the unital algebra generated by $\{1, w_j, w_j^t\}$. Then the algebras $\{ \cA_j \}_{j \in J}$ are free. In addition, when $r^2 = 0$, we have that $w_j$ and $w_j^t$ are free for each $j$.
\end{theorem}

\begin{remark}\label{rem:f notation and cumulants}
As we saw in Equation \eqref{eq:a expression}, for $\pi \in \DT(n)$, we have
\[
f(\pi, \bm\epsilon) = \kappa_{\overline\pi}
(w^{(\epsilon_1)}, \dots, w^{(\epsilon_n)}).
\]
\end{remark}

\subsection{The sub-leading term and $B_n$}

\begin{definition}\label{def:U_m_k}
For $m \in \set{0, 1}$ and $k \geq 1$, we say that $G(\pi)$ is a \textit{$(m,k)$-unicyclic} graph if one of the following conditions is satisfied:
\begin{itemize}
    \item \textbf{Case I:} $k = 1$ and $m = 0$. In this case the graph $\overline G(\pi)$ consists of a tree with a single loop attached to a vertex, i.e. $\overline G(\pi)$ is a unicyclic graph whose unique cycle has length 1. The partition $\overline\pi$ is a pairing such that the paired edges have opposite orientations.

    \item \textbf{Case II:} $k = 2$ and $m = 0$. In this case the graph $\overline G(\pi)$ is a tree. The partition $\overline\pi$ is a pairing, except for one block of size 4, such that the paired edges have opposite orientations. The block of size 4 must be decomposable into two pairs of edges in which each pair consists of edges with opposite orientations. 

    \item \textbf{Case III:} $k \geq 3$. In this case the graph $\overline G(\pi)$ is a unicyclic graph whose unique cycle has length $k$. The partition $\overline\pi$ is a pairing. For blocks of $\overline\pi$ not associated with the unique cycle in $\overline G(\pi)$, the paired edges have opposite orientations. When $m = 0$, each pair within the cycle consists of edges with opposite orientations; when $m = 1$, 
    they consist of edges with the same orientation.  
\end{itemize}
We call $m$ the \textit{orientation parameter.}
\end{definition}

We prove, subject to certain constraints, that the only partitions $\pi$ with $\#(\pi) = \frac{n}{2}$ and $E(\pi, \bm j, \bm \epsilon) \neq 0$ are those whose associated graph $G(\pi)$ is a $(m,k)$-unicyclic graph for some $m$ and some $k$.

\begin{notation}\label{notation:U_m_k}
We denote by $\U_{m,k}(S) \subseteq \mathcal{P}(S)$ the set of partitions $\pi$ whose associated graph $G(\pi)$ is an $(m,k)$-unicyclic graph.
\begin{itemize}
    \item For $\pi \in \U_{0,1}(n)$, let 
    \[
    \tilde{f}(\pi, \bm \epsilon) = s^2 r^{2p} \sigma^{2q},
    \] 
    where $p$ is the number of pairs $(u, v)$ of $\overline\pi$ such that $\epsilon_u = -\epsilon_v$ and $u \not\sim_\pi v$; and $q$ is the number of pairs $(u, v) \in \overline\pi$ such that $\epsilon_u  = \epsilon_v$ and $u \not\sim_\pi v$.
    \item For $\pi \in \U_{0,k}(n)$, with $k \geq 3$, let 
    \[
    f(\pi, \bm \epsilon) = r^{2p} \sigma^{2 q},
    \] 
    where $p$ is the number of pairs $(u, v) \in \overline\pi$ with $\epsilon_u = - \epsilon_v$; and $q$ is the number of pairs $(u, v) \in \overline\pi$ with $\epsilon_u =  \epsilon_v$.
    \item For $\pi \in \U_{1,k}(n)$, with $k \geq 3$, let 
    \[
    g(\pi, \bm \epsilon) = r^{2p} \sigma^{2q}, 
    \]
    where $p$ is the number of pairs $(u, v) \in \overline{\pi}$ with $\epsilon_u = -\epsilon_v$ and $u, v$ have the opposite orientation, or $\epsilon_u = \epsilon_v$ and $u, v$ the same orientation; and $q$ is the number of pairs $(u,v) \in \overline{\pi}$ with $\epsilon_u = \epsilon_v$ and $u, v$ have the opposite orientation, or $\epsilon_u = -\epsilon_v$ and $u, v$ the same orientation.
\end{itemize}
\end{notation}


\begin{lemma}\label{lemma:double-cycle}
Let $n \geq 6$ and let $\pi \in \PP(n)$ be such that $\overline{\pi}$ is a pairing and $\overline{G}(\pi)$ is a cycle. If $\pi$ does not have a turn-around point, then 
\[
    \pi = \overline{\pi} = \set{\set{1, n/2+1}, \set{2, n/2+2}, \ldots, \set{n/2, n}},
\]
and $\pi \in \U_{1, n/2}(n)$. If $\pi$ has a turn-around point, then
\[
    \pi = \set{\set{k, k + n/2}} \cup \set{\set{u, v} : u,v \in [n],\, u \neq v,\, u + v \equiv 2k\, (\operatorname{mod}\, n)},
\]
and
\[
    \overline\pi = \set{\set{u, v} : u,v \in [n],\, u + v \equiv 2k-1\, (\operatorname{mod}\, n)},
\]
where $k$ is the smallest element in $[n]$ such that $k$ is a turn-around point in its corresponding block $[k]_\pi$. Moreover, $\pi \in \U_{0,n/2}(n)$.
\end{lemma}

\begin{proof}
Because $\overline{G}(\pi)$ is a cycle, $\overline\pi$ is a pairing and $\gamma^{-1}$ is an Eulerian circuit on $G(\pi)$, the construction of $G(\pi)$ and the presence (or not) of a turn-around point force the partitions $\pi$ and $\overline\pi$ in each case.
\end{proof}

Let $\pi \in \PP(n)$ be a partition such that $\#(\pi) = \frac{n}{2}$. If $E(\pi, \bm j, \bm \epsilon) \neq 0$, we know by \eqref{eq:edge upper bound} that
\[
    \frac{n}{2} - 1 \leq \#(\overline\pi) \leq \frac{n}{2}.
\]
We treat the cases $\#(\overline\pi) = \frac{n}{2} - 1$ and $\#(\overline\pi) = \frac{n}{2}$ separately in the following lemmas.

\begin{notation}\label{notation:fourth infinitesimal cumulant}
In the next Lemma we shall evaluate $E(\pi, \bm j, \bm \epsilon)$ for the $\pi$'s that appear at the sub-leading level. The dependence of $E(\pi, \bm j, \bm \epsilon)$ on $\bm \epsilon$ will require some additional notation. Let 
\begin{center}\renewcommand{\arraystretch}{1.2} 
\begin{tabular}{rcl}
$\alpha$ & = & $\esp ( W_{12} \overline{W_{12}} W_{12} \overline{W_{12}}) $, \\
$\beta$ & = & $\frac{1}{2} \esp ( W_{12} W_{12} W_{12} W_{12} + 
\overline{W_{12}}\overline{W_{12}} 
\overline{W_{12}}\overline{W_{12}}) $, \\
$\rho$ & = & $\frac{1}{2}\esp ( W_{12}  W_{12} W_{12}\overline{W_{12}} +
W_{12}\overline{W_{12}} \overline{W_{12}} \overline{W_{12}} ) $. 
\end{tabular}
\end{center}
In addition we need the following subsets of $\cP(4)$. 
\begin{align*}
A = \ &  \big\{\, 
\begin{tikzpicture}[x=0.65cm,y=0.65cm,baseline=0.15cm,line cap=round,line join=round]
  \tikzset{partline/.style={line width=0.45pt}}
  \coordinate (p1) at (0,0);
  \coordinate (p2) at (1,0);
  \coordinate (p3) at (2,0);
  \coordinate (p4) at (3,0);
  \draw[partline] (0,0.75) -- (0,0.18);
  \draw[partline] (1,0.75) -- (1,0.18);
  \draw[partline] (2,0.75) -- (2,0.18);
  \draw[partline] (3,0.75) -- (3,0.18);
  \draw[partline] (0,0.18) -- (3,0.18);
\end{tikzpicture}\, , \ 
\begin{tikzpicture}[x=0.65cm,y=0.65cm,baseline=0.15cm,line cap=round,line join=round]
  \tikzset{partline/.style={line width=0.45pt}}
  \coordinate (p1) at (0,0);
  \coordinate (p2) at (1,0);
  \coordinate (p3) at (2,0);
  \coordinate (p4) at (3,0);
  \draw[partline] (0,0.75) -- (0,0.18);
  \draw[partline] (1,0.75) -- (1,0.18);
  \draw[partline] (0,0.18) -- (1,0.18);
  \draw[partline] (2,0.75) -- (2,0.18);
  \draw[partline] (3,0.75) -- (3,0.18);
  \draw[partline] (2,0.18) -- (3,0.18);
\end{tikzpicture}\, , \
\begin{tikzpicture}[x=0.65cm,y=0.65cm,baseline=0.15cm,line cap=round,line join=round]
  \tikzset{partline/.style={line width=0.45pt}}
  \coordinate (p1) at (0,0);
  \coordinate (p2) at (1,0);
  \coordinate (p3) at (2,0);
  \coordinate (p4) at (3,0);
  \draw[partline] (0,0.75) -- (0,0.18);
  \draw[partline] (3,0.75) -- (3,0.18);
  \draw[partline] (0,0.18) -- (3,0.18);
  \draw[partline] (1,1.03) -- (1,0.46);
  \draw[partline] (2,1.03) -- (2,0.46);
  \draw[partline] (1,0.46) -- (2,0.46);
\end{tikzpicture}  \,\big\} \\
B = \ &  \big\{\,
\begin{tikzpicture}[x=0.65cm,y=0.65cm,baseline=0.3cm,line cap=round,line join=round]
  \tikzset{partline/.style={line width=0.45pt}}
  \coordinate (p1) at (0,0);
  \coordinate (p2) at (1,0);
  \coordinate (p3) at (2,0);
  \coordinate (p4) at (3,0);
  \draw[partline] (0,1.03) -- (0,0.46);
  \draw[partline] (2,1.03) -- (2,0.46);
  \draw[partline] (0,0.46) -- (2,0.46);
  \draw[partline] (1,0.75) -- (1,0.18);
  \draw[partline] (3,0.75) -- (3,0.18);
  \draw[partline] (1,0.18) -- (3,0.18);
\end{tikzpicture}
\,\big\} \\
C = \  &  \big\{\, \begin{tikzpicture}[x=0.65cm,y=0.65cm,baseline=0.15cm,line cap=round,line join=round]
  \tikzset{partline/.style={line width=0.45pt}}
  \coordinate (p1) at (0,0);
  \coordinate (p2) at (1,0);
  \coordinate (p3) at (2,0);
  \coordinate (p4) at (3,0);
  \draw[partline] (0,0.75) -- (0,0.25);
  \draw[partline] (1,0.75) -- (1,0.18);
  \draw[partline] (2,0.75) -- (2,0.18);
  \draw[partline] (3,0.75) -- (3,0.18);
  \draw[partline] (1,0.18) -- (3,0.18);
\end{tikzpicture} \ , \ 
\begin{tikzpicture}[x=0.65cm,y=0.65cm,baseline=0.15cm,line cap=round,line join=round]
  \tikzset{partline/.style={line width=0.45pt}}
  \coordinate (p1) at (0,0);
  \coordinate (p2) at (1,0);
  \coordinate (p3) at (2,0);
  \coordinate (p4) at (3,0);
  \draw[partline] (0,0.75) -- (0,0.18);
  \draw[partline] (2,0.75) -- (2,0.18);
  \draw[partline] (3,0.75) -- (3,0.18);
  \draw[partline] (0,0.18) -- (3,0.18);
  \draw[partline] (1,0.75) -- (1,0.35);
\end{tikzpicture}\, , \ 
\begin{tikzpicture}[x=0.65cm,y=0.65cm,baseline=0.15cm,line cap=round,line join=round]
  \tikzset{partline/.style={line width=0.45pt}}
  \coordinate (p1) at (0,0);
  \coordinate (p2) at (1,0);
  \coordinate (p3) at (2,0);
  \coordinate (p4) at (3,0);
  \draw[partline] (0,0.75) -- (0,0.18);
  \draw[partline] (1,0.75) -- (1,0.18);
  \draw[partline] (3,0.75) -- (3,0.18);
  \draw[partline] (0,0.18) -- (3,0.18);
  \draw[partline] (2,0.75) -- (2,0.35);
\end{tikzpicture}\ , \ 
\begin{tikzpicture}[x=0.65cm,y=0.65cm,baseline=0.15cm,line cap=round,line join=round]
  \tikzset{partline/.style={line width=0.45pt}}
  \coordinate (p1) at (0,0);
  \coordinate (p2) at (1,0);
  \coordinate (p3) at (2,0);
  \coordinate (p4) at (3,0);
  \draw[partline] (0,0.75) -- (0,0.18);
  \draw[partline] (1,0.75) -- (1,0.18);
  \draw[partline] (2,0.75) -- (2,0.18);
  \draw[partline] (0,0.18) -- (2,0.18);
  \draw[partline] (3,0.75) -- (3,0.25);
\end{tikzpicture} \  
\,\big\},
\end{align*}
\[
A' = \   \big\{\, 
\begin{tikzpicture}[x=0.65cm,y=0.65cm,baseline=0.15cm,line cap=round,line join=round]
  \tikzset{partline/.style={line width=0.45pt}}
  \coordinate (p1) at (0,0);
  \coordinate (p2) at (1,0);
  \coordinate (p3) at (2,0);
  \coordinate (p4) at (3,0);
  \draw[partline] (0,0.75) -- (0,0.18);
  \draw[partline] (1,0.75) -- (1,0.18);
  \draw[partline] (2,0.75) -- (2,0.18);
  \draw[partline] (3,0.75) -- (3,0.18);
  \draw[partline] (0,0.18) -- (3,0.18);
\end{tikzpicture}\, , \ 
\begin{tikzpicture}[x=0.65cm,y=0.65cm,baseline=0.15cm,line cap=round,line join=round]
  \tikzset{partline/.style={line width=0.45pt}}
  \coordinate (p1) at (0,0);
  \coordinate (p2) at (1,0);
  \coordinate (p3) at (2,0);
  \coordinate (p4) at (3,0);
  \draw[partline] (0,0.75) -- (0,0.18);
  \draw[partline] (1,0.75) -- (1,0.18);
  \draw[partline] (0,0.18) -- (1,0.18);
  \draw[partline] (2,0.75) -- (2,0.18);
  \draw[partline] (3,0.75) -- (3,0.18);
  \draw[partline] (2,0.18) -- (3,0.18);
\end{tikzpicture}\  \,\big\},
\]
\[
A'' = \ \big\{\, 
\begin{tikzpicture}[x=0.65cm,y=0.65cm,baseline=0.15cm,line cap=round,line join=round]
  \tikzset{partline/.style={line width=0.45pt}}
  \coordinate (p1) at (0,0);
  \coordinate (p2) at (1,0);
  \coordinate (p3) at (2,0);
  \coordinate (p4) at (3,0);
  \draw[partline] (0,0.75) -- (0,0.18);
  \draw[partline] (1,0.75) -- (1,0.18);
  \draw[partline] (2,0.75) -- (2,0.18);
  \draw[partline] (3,0.75) -- (3,0.18);
  \draw[partline] (0,0.18) -- (3,0.18);
\end{tikzpicture}\, , \ 
\begin{tikzpicture}[x=0.65cm,y=0.65cm,baseline=0.15cm,line cap=round,line join=round]
  \tikzset{partline/.style={line width=0.45pt}}
  \coordinate (p1) at (0,0);
  \coordinate (p2) at (1,0);
  \coordinate (p3) at (2,0);
  \coordinate (p4) at (3,0);
  \draw[partline] (0,0.75) -- (0,0.18);
  \draw[partline] (3,0.75) -- (3,0.18);
  \draw[partline] (0,0.18) -- (3,0.18);
  \draw[partline] (1,1.03) -- (1,0.46);
  \draw[partline] (2,1.03) -- (2,0.46);
  \draw[partline] (1,0.46) -- (2,0.46);
\end{tikzpicture}\  \,\big\},
\]
\[
A''' = \ \big\{\, 
\begin{tikzpicture}[x=0.65cm,y=0.65cm,baseline=0.15cm,line cap=round,line join=round]
  \tikzset{partline/.style={line width=0.45pt}}
  \coordinate (p1) at (0,0);
  \coordinate (p2) at (1,0);
  \coordinate (p3) at (2,0);
  \coordinate (p4) at (3,0);
  \draw[partline] (0,0.75) -- (0,0.18);
  \draw[partline] (1,0.75) -- (1,0.18);
  \draw[partline] (0,0.18) -- (1,0.18);
  \draw[partline] (2,0.75) -- (2,0.18);
  \draw[partline] (3,0.75) -- (3,0.18);
  \draw[partline] (2,0.18) -- (3,0.18);
\end{tikzpicture}\, , \    
\begin{tikzpicture}[x=0.65cm,y=0.65cm,baseline=0.15cm,line cap=round,line join=round]
  \tikzset{partline/.style={line width=0.45pt}}
  \coordinate (p1) at (0,0);
  \coordinate (p2) at (1,0);
  \coordinate (p3) at (2,0);
  \coordinate (p4) at (3,0);
  \draw[partline] (0,0.75) -- (0,0.18);
  \draw[partline] (3,0.75) -- (3,0.18);
  \draw[partline] (0,0.18) -- (3,0.18);
  \draw[partline] (1,1.03) -- (1,0.46);
  \draw[partline] (2,1.03) -- (2,0.46);
  \draw[partline] (1,0.46) -- (2,0.46);
\end{tikzpicture}\  \,\big\},
\]
\[
B' = \ \big\{\, 
\begin{tikzpicture}[x=0.65cm,y=0.65cm,baseline=0.3cm,line cap=round,line join=round]
  \tikzset{partline/.style={line width=0.45pt}}
  \coordinate (p1) at (0,0);
  \coordinate (p2) at (1,0);
  \coordinate (p3) at (2,0);
  \coordinate (p4) at (3,0);
  \draw[partline] (0,1.03) -- (0,0.46);
  \draw[partline] (2,1.03) -- (2,0.46);
  \draw[partline] (0,0.46) -- (2,0.46);
  \draw[partline] (1,0.75) -- (1,0.18);
  \draw[partline] (3,0.75) -- (3,0.18);
  \draw[partline] (1,0.18) -- (3,0.18);
\end{tikzpicture}\, , \ 
\begin{tikzpicture}[x=0.65cm,y=0.65cm,baseline=0.15cm,line cap=round,line join=round]
  \tikzset{partline/.style={line width=0.45pt}}
  \coordinate (p1) at (0,0);
  \coordinate (p2) at (1,0);
  \coordinate (p3) at (2,0);
  \coordinate (p4) at (3,0);
  \draw[partline] (0,0.75) -- (0,0.18);
  \draw[partline] (3,0.75) -- (3,0.18);
  \draw[partline] (0,0.18) -- (3,0.18);
  \draw[partline] (1,1.03) -- (1,0.46);
  \draw[partline] (2,1.03) -- (2,0.46);
  \draw[partline] (1,0.46) -- (2,0.46);
\end{tikzpicture}\  \,\big\}.
\]
\[
B'' = \   \big\{\, 
\begin{tikzpicture}[x=0.65cm,y=0.65cm,baseline=0.3cm,line cap=round,line join=round]
  \tikzset{partline/.style={line width=0.45pt}}
  \coordinate (p1) at (0,0);
  \coordinate (p2) at (1,0);
  \coordinate (p3) at (2,0);
  \coordinate (p4) at (3,0);
  \draw[partline] (0,1.03) -- (0,0.46);
  \draw[partline] (2,1.03) -- (2,0.46);
  \draw[partline] (0,0.46) -- (2,0.46);
  \draw[partline] (1,0.75) -- (1,0.18);
  \draw[partline] (3,0.75) -- (3,0.18);
  \draw[partline] (1,0.18) -- (3,0.18);
\end{tikzpicture}\, , \ 
\begin{tikzpicture}[x=0.65cm,y=0.65cm,baseline=0.15cm,line cap=round,line join=round]
  \tikzset{partline/.style={line width=0.45pt}}
  \coordinate (p1) at (0,0);
  \coordinate (p2) at (1,0);
  \coordinate (p3) at (2,0);
  \coordinate (p4) at (3,0);
  \draw[partline] (0,0.75) -- (0,0.18);
  \draw[partline] (1,0.75) -- (1,0.18);
  \draw[partline] (0,0.18) -- (1,0.18);
  \draw[partline] (2,0.75) -- (2,0.18);
  \draw[partline] (3,0.75) -- (3,0.18);
  \draw[partline] (2,0.18) -- (3,0.18);
\end{tikzpicture}\  \,\big\},
\]
\[
B''' = \   \big\{\, 
\begin{tikzpicture}[x=0.65cm,y=0.65cm,baseline=0.3cm,line cap=round,line join=round]
  \tikzset{partline/.style={line width=0.45pt}}
  \coordinate (p1) at (0,0);
  \coordinate (p2) at (1,0);
  \coordinate (p3) at (2,0);
  \coordinate (p4) at (3,0);
  \draw[partline] (0,1.03) -- (0,0.46);
  \draw[partline] (2,1.03) -- (2,0.46);
  \draw[partline] (0,0.46) -- (2,0.46);
  \draw[partline] (1,0.75) -- (1,0.18);
  \draw[partline] (3,0.75) -- (3,0.18);
  \draw[partline] (1,0.18) -- (3,0.18);
\end{tikzpicture}\, , \
\begin{tikzpicture}[x=0.65cm,y=0.65cm,baseline=0.15cm,line cap=round,line join=round]
  \tikzset{partline/.style={line width=0.45pt}}
  \coordinate (p1) at (0,0);
  \coordinate (p2) at (1,0);
  \coordinate (p3) at (2,0);
  \coordinate (p4) at (3,0);
  \draw[partline] (0,0.75) -- (0,0.18);
  \draw[partline] (1,0.75) -- (1,0.18);
  \draw[partline] (2,0.75) -- (2,0.18);
  \draw[partline] (3,0.75) -- (3,0.18);
  \draw[partline] (0,0.18) -- (3,0.18);
\end{tikzpicture}
\  \,\big\}.
\]

\end{notation}

\begin{lemma}\label{lemma:2-4 trees}
Let $n \geq 4$ and let $\pi \in \PP(n)$ be a partition with $\#(\pi) = \frac{n}{2}$ and $\#(\overline\pi) = \frac{n}{2} - 1$. If $E(\pi, \bm j, \bm \epsilon) \neq 0$, then
$\pi \in \U_{0,2}(n)$ with the unique block of size four $B \in \overline{\pi}$ satisfying $\overline\pi \setminus \{B\} \leq \ker(\bm j_1)$ and $\ker(\bm j_2)$ is either a pairing or the partition $1_4$, where $\bm j_1 = \bm j|_{[n] \setminus B}$, $\bm j_2 = \bm j|_B$. Moreover, $\overline{\pi} \in \NC(n)$.

Let $\bm \epsilon_2 = \bm \epsilon|_B$ and let $p$ be the number of pairs $(u, v) \in \overline \pi$, such that $\epsilon_u = \epsilon_v$ and $q$ be the number of pairs for which $\epsilon_u = - \epsilon_v$. 

\begin{enumerate}
\item
If $\ker(\bm j_2) = 1_4$, then
\[
E(\pi, \bm j, \bm \epsilon) =
\sigma^{2p} r^{2q} \times
\begin{cases}
\alpha & \ker(\bm \epsilon_2) \in A \\
\beta  & \ker(\bm \epsilon_2) \in B \\
\rho   & \ker(\bm \epsilon_2) \in C 
\end{cases}.
\]
\item
If $\ker(\bm j_2) = \big\{ \{1, 2\}, \{3, 4\} \big\}$, then
\[
E(\pi, \bm j, \bm \epsilon) = 
\sigma^{2p} r^{2q} \times
\begin{cases}
\sigma^4 & \ker(\bm \epsilon_2) \in A' \\
r^4  & \ker(\bm \epsilon_2) \in B' \\
\sigma^2 r^2   & \ker(\bm \epsilon_2) \in C 
\end{cases}.
\]
\item
If $\ker(\bm j_2) = \{\{1, 4\}, \{2, 3\}\}$, then
\[
E(\pi, \bm j, \bm \epsilon) =
\sigma^{2p} r^{2q} \times
\begin{cases}
\sigma^4 & \ker(\bm \epsilon_2) \in A'' \\
r^4  & \ker(\bm \epsilon_2) \in B'' \\
\sigma^2 r^2   & \ker(\bm \epsilon_2) \in C 
\end{cases}.
\]

\item
If $\ker(\bm j_2) = \{\{1, 3\}, \{2, 4\}\}$, then
\[
E(\pi, \bm j, \bm \epsilon) =
\sigma^{2p} r^{2q} \times
\begin{cases}
\sigma^4 & \ker(\bm \epsilon_2) \in A''' \\
r^4  & \ker(\bm \epsilon_2) \in B''' \\
\sigma^2 r^2   & \ker(\bm \epsilon_2) \in C 
\end{cases}.
\]

\end{enumerate}
\end{lemma}

\begin{proof}
If $n = 4$, the only partition $\pi \in \PP(n)$ with $\#(\pi) = \frac{n}{2}$ and $\#(\overline{\pi}) = \frac{n}{2} - 1$ is $\pi = \set{\set{1,3}, \set{2,4}}$. Therefore, the graph $G(\pi)$ is composed of two vertices joined by two edges in one orientation (the edges $e_1$, and $e_3$) and two edges in the opposite orientation (the edges $e_2$, and $e_4$), see Figure \ref{fig:block of size four}. Hence, the graph $G(\pi)$ is a $(0,2)$-unicyclic graph. Moreover $\overline{\pi} = 1_4 \in \NC(n)$. Now let us consider the values of $E(\pi, \bm j, \bm \epsilon)$.

There are sixteen possible $\epsilon$'s. Suppose $i_1 < i_2$. Then, after considering all 16 cases, there are only $3$ possible values of $E(\pi, \bm j,\ab \bm \epsilon)$, which only depend on $\ker(\bm \epsilon)$. They are $\alpha$, $\beta$ or $\rho$ and occur exactly when $\ker(\epsilon)$ is in any of $A$, $B$ or $C$. The same calculation also shows that when $\ker(\bm j_2)$ is a pairing the only possible values of $E(\pi, \bm j, \bm \epsilon)$ are $\sigma^4$, $r^4$ or $\sigma^2 r^2$. Which one we get depends both on $\ker(\bm j)$ and $\ker(\bm \epsilon)$, as listed in the statement of the Lemma. 
In this case, $\overline \pi$ has only one block (the block $B$), so $ p = q = 0$.                          

\begin{figure}[t]
\begin{center}\includegraphics[width=15em]{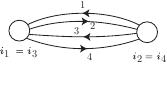}\end{center}
\caption{\small\raggedright\label{fig:block of size four} When $ n = 4$ the only possible $\pi$ is $\{\{1, 3\}, \{2, 4\}\}$.}
\end{figure}

Suppose that  $n \geq 6$. Let $\pi \in \PP(n)$ be a partition with $\#(\pi) = \frac{n}{2} - 1$ and $\#(\overline\pi) = \frac{n}{2}$. Then, by Equation \eqref{eq:edge upper bound}, $\overline{G}(\pi)$ is a tree.
Moreover, since $G(\pi)$ admits an Eulerian circuit (given by $\gamma^{-1}$), all the vertices of $G(\pi)$ have even total degree.\footnote{If a directed graph has an Eulerian circuit, then the number of incoming edges equals number of outgoing edges for every vertex.} Therefore, 
$\overline \pi$ must have a block $B$ of size $4$, and all others of size $2$. Following the proof of Lemma~\ref{lemma:double-trees}, we find in the same way that the paired edges of $\overline{\pi}$ have opposite orientations and the block of size four contains two edges in one orientation and two edges in the opposite orientation (as in the base case $n = 4$). Moreover, it follows that $\overline\pi \in \NC(n)$.
Since $E(\pi, \bm j, \bm \epsilon) \not = 0$, we must have $\overline\pi \setminus \{ B \} \leq \ker(\bm j_1)$. Thus, the contribution of each block $(u, v)$ of $\overline \pi$, not equal to $B$, is either $\sigma^2$ or $r^2$ depending on whether $\epsilon_u = \epsilon_v$ or $\epsilon_u = -\epsilon_v$ (exactly as in Lemma~\ref{lemma:double-trees}). The contribution of the block $B$ is exactly as in the case when $n = 4$. 
\end{proof}


\begin{remark}
Suppose that 
$\pi \in \U_{0,2}(n)$. If $\overline\pi \setminus \{B\} \leq \ker(\bm j_1)$ and $\ker(\bm j_2)$ is either a pairing or the partition $1_4$, then $E(\pi, \bm j, \bm \epsilon)$ is as in the cases $(i)$, $(ii)$, $(iii)$ and $(iv)$ above. 
Each block of size $2$ of $\overline \pi$ contributes to $E(\pi, \bm j, \bm \epsilon)$ either $\sigma^2$ or $r^2$ and the contribution from $B$ is any of $\alpha$, $\beta$, $\rho$, $\sigma^4$, $r^2$ or $\sigma^2 r^2$. 
However, we might have $\beta = 0$, $\rho = 0$ or $r^2 = 0$ and then $E(\pi, \bm j, \bm \epsilon)$ can be $0$.
\end{remark}

\begin{lemma}\label{lemma:m-k unicycle}
Let $n \geq 2$ and let $\pi \in \PP(n)$ be a partition with $\#(\pi) = \#(\overline\pi) =  n/2$ and $E(\pi, \bm j, \bm \epsilon) \neq 0$.

\smallskip\noindent
\textit{Case} $(a)$: 
Suppose $r^2 \not = 0$. Then  $\pi \in \U_{m,k}(n)$, for some $m \in \set{0,1}$ and some $k \neq 2$ with $\overline\pi \leq \ker(\bm j)$.

\smallskip\noindent
\textit{Case} $(b)$:  Suppose $r^2 = 0$. Then either:
\begin{itemize}
    \item $\pi \in \U_{0,1}(n)$, $\overline\pi \leq \ker(\bm j)$ and $\overline\pi|_{[n] \setminus V} \leq \ker(\epsilon)|_{[n] \setminus V}$, where $V \in \overline\pi$ is the corresponding block of loops in $G(\pi)$; or
    \item $\pi \in \U_{0,k}(n)$, for some $k \geq 3$ and $\overline\pi \leq \ker(\bm j) \wedge \ker(\bm \epsilon)$; or
    \item $\pi \in \U_{1,k}(n)$, for some $k \geq 3$, $\overline\pi \leq \ker(\bm j)$ and, for every $\set{u,v} \in \overline\pi$,
    \[
        \begin{cases}
            \epsilon_u = \epsilon_v & \text{if } u,v \text{ have the opposite orientation},\\
            \epsilon_u = -\epsilon_v & \text{if } u,v \text{ have the same orientation}.
        \end{cases}
    \]
\end{itemize}
Under either case $(a)$ or $(b)$:
\[
    E(\pi, \bm j, \bm \epsilon) = \begin{cases}
        \tilde f(\pi, \bm \epsilon) & \text{if } m = 0,\, k = 1,\\ 
        f(\pi, \bm \epsilon) & \text{if } m = 0,\, k \geq 3,\\
        g(\pi, \bm \epsilon) & \text{if } m = 1,\, k \geq 3.
    \end{cases}
\]
Moreover, $\overline\pi \in \NC_2(n)$ if $m = 0$, and $\overline{\pi} \notin \NC(n)$ if $m = 1$.    
\end{lemma}

\begin{proof}
We proceed by induction on $n$, but first we deal with the case when $\overline{G}(\pi)$ has no leaves. If $n = 2$, the only partition $\pi \in \PP(n)$ such that $\#(\pi) = \#(\overline\pi) = \frac{n}{2}$ is $\pi = \set{\{1,2\}}$. Therefore, the graph $G(\pi)$ consists of a single vertex and two loops. Then, $\overline{G}(\pi)$ is a $(0,1)$-unicyclic graph and $\overline{\pi} = 1_2 \in NC_2(n)$. Since we are on the diagonal, the value of $\bm\epsilon$ does not matter and
\[
    E(\pi, \bm j, \bm \epsilon) = \begin{cases}
        s^2 & \text{if } \ker(\bm j) = 1_2,\\
        0 &   \text{if } \ker(\bm j) = 0_2.
    \end{cases}
\]
Hence, $E(\pi, \bm j, \bm \epsilon) \neq 0$ if and only if $\overline\pi \leq \ker(\bm j)$.  This shows that the conclusion holds in either Case $(a)$ or $(b)$. Next suppose $n \geq 4$. 

\medskip\noindent
\textit{Case} $(a)$: Suppose $E(\pi, \bm j, \bm \epsilon) \neq 0$ and $r^2 \neq 0$.

Since $\#(\pi) = \#(\overline\pi)$ we have that the number of vertices and the number of edges in $\overline{G}(\pi)$ are equal. Thus $\overline{G}(\pi)$ has exactly one cycle (in the sense of homology). By the construction of $\overline{G}(\pi)$, the length of the cycle cannot be $2$. We have already dealt with the case of a cycle of length $1$. So the length of the cycle must be at least $3$ and $n \geq 6$. By Lemma \ref{lemma:edge upper bound}, $\overline\pi$ is a pairing. 

If $\overline{G}(\pi)$ has no leaves, then $\overline{G}(\pi)$ is a cycle of length $n/2$. Because of Lemma \ref{lemma:double-cycle}, if $\pi$ does not have a turn-around point, then $\pi \in \U_{1,n/2}(n)$ with $\overline{\pi} \notin \NC(n)$. As $E(\pi, \bm j, \bm \epsilon) \not = 0$, we must have $\overline\pi \leq \ker(\bm j)$ and for every $(u, v)\in \overline\pi$, $u$ and $v$ have the same orientation. Thus $E(\pi, \bm j, \bm \epsilon)$ is $r^{2p}\sigma^{2q}$ where $p$ is the number of $(u, v) \in \overline \pi$ such that $\epsilon_u = \epsilon_v$, and $q$ is the number such that $\epsilon_u = - \epsilon_v$. This is exactly $g(\pi, \bm \epsilon)$.

If $\pi$ has a turn-around point, then by Lemma \ref{lemma:double-cycle}, $\pi \in \U_{0,n/2}(n)$ with $\overline{\pi} \in \NC_2(n)$. In this first case, we must have $u$ and $v$ have opposite orientations for all $(u, v) \in \overline \pi$, and so $E(\pi, \bm j, \bm \epsilon) = f(\pi, \bm \epsilon)$. 

Now suppose $\overline{G}(\pi)$ has a leaf $V$. Since $\overline\pi$ is a pairing, exactly two edges have $V$ as an endpoint in $G(\pi)$. By the construction of $G(\pi)$, we obtain that $V = \set{k}$ is a singleton such that $\set{\gamma^{-1}(k), k} \in \overline{\pi}$. Then, by Lemma \ref{lemma:pruning}, it follows that $\set{k}$ is a double-leaf and $\overline{\pi} = \overline\pi_1 \cup \overline\pi_2$, where
\[
    \pi_1 = \set{\set{\gamma^{-1}(k)}, \set{k}}, \qquad \pi_2 = \pi|_{[n+2] \setminus \set{\gamma^{-1}(k), k}}.
\]
The contribution of the edge $(\gamma^{-1}(k), k) \in \overline\pi$ to $E(\pi, \bm j, \bm \epsilon)$ is $r^2$ if $\epsilon_{\gamma^{-1}(k)} = \epsilon_{-k}$, and $\sigma^2$ if $\epsilon_{\gamma^{-1}(k)} = -\epsilon_{-k}$. Thus
\[
E(\pi, \bm j, \bm \epsilon)
=
E(\pi_2, \bm j_2, \bm\epsilon_2) \times
\begin{cases} r^2 & \epsilon_{\gamma^{-1}(k)} = \epsilon_{-k}, \\
         \sigma^2 & \epsilon_{\gamma^{-1}(k)} = -\epsilon_{-k}.
\end{cases}
\]
By induction, as in the proof of Lemma~\ref{lemma:double-trees}, we keep removing leaves until we reach the case where there are no leaves, and then we have already established the lemma in these cases.

\medskip\noindent
\textit{Case} $(b)$: Suppose $E(\pi, \bm j, \bm \epsilon) \not = 0$ and $r^2 = 0$.

We cannot have a $(u, v) \in \overline\pi$ with $\epsilon_u = \epsilon_v$, as it produces a factor of $\esp(W_{12}^2) = r^2$. In addition, by independence, we must have $\overline \pi \leq \ker(\bm j)$. When $\pi \in \U_{0,1}(n)$, this means  we must have  $\epsilon_u = \epsilon_v$  for all edges $(u, v)$ not in the loop. For $\pi \in \U_{0, k}$, with $k \geq 3$, we must have $\epsilon_u = \epsilon_v$ for all edges $(u, v)$: i.e. $\overline\pi \leq \ker(\bm j) \wedge \ker( \bm \epsilon)$. For $\pi \in \U_{1, k}$, for $k \geq 3$, this means $\epsilon_u = \epsilon_v$ for all $(u, v)$ not in the cycle and $\epsilon_u = - \epsilon_v$ for all $(u, v) $ in the cycle. This is exactly what is claimed in case $(b)$. 
\end{proof}

Once we know that $\pi \in \U_{m, k}(n)$, then the proof of the Lemma above shows us how to compute $E(\pi, \bm j, \bm \epsilon)$. We restate this in the following Remark.

\begin{remark}\label{corollary:m-k unicycle-converse}
Let $n \geq 2$ and let $\pi \in \PP(n)$ be a partition with $\#(\pi) = \#(\overline\pi) =  n/2$.  If $\pi \in \U_{m,k}(n)$, for some $m \in \set{0,1}$ and some $k \neq 2$, with $\overline\pi \leq \ker(\bm j)$ then
\[
    E(\pi, \bm j, \bm \epsilon) = \begin{cases}
        \tilde f(\pi, \bm \epsilon) & \text{if } m = 0,\, k = 1,\\ 
        f(\pi, \bm \epsilon) & \text{if } m = 0,\, k \geq 3,\\
        g(\pi, \bm \epsilon) & \text{if } m = 1,\, k \geq 3.
    \end{cases}
\]
\end{remark}

\begin{corollary}\label{cor:m-k unicycle}
Every $\pi \in \U_{0,k}(n)$ satisfies $\overline\pi \in \NC(n)$. Furthermore, if $k \neq 2$, then $\overline{\pi} \in \NC_2(n)$.
\end{corollary}

\begin{proof}
The result follows because of Lemma \ref{lemma:2-4 trees} and Lemma \ref{lemma:m-k unicycle} by setting $\bm j = \bm \epsilon = (1,\ldots, 1)$ and $r^2 = \sigma^2$. 
\end{proof}

\begin{remark}\label{rmk:mapping_Q}
Let $\pi \in \DT(n)$ and let $U, V \in \pi$ be vertices at distance $k$ in $\overline G(\pi)$. We denote by $Q_{\pi}(U, V)$ the partition resulting from merging the blocks $U$ and $V$. Note that $G(Q_{\pi}(U, V))$ is the graph obtained by identifying the vertices $U$ and $V$ in $G(\pi)$.

Since $G(\pi)$ is a double-tree, we have the following cases when identifying two vertices $U, V \in \pi$ at distance $k$ in $G(\pi)$:
\begin{itemize}
    \item When $k = 2$, this operation creates exactly one edge in $\overline G(Q_{\pi}(U, V))$ whose corresponding block in $\overline{Q_{\pi}(U, V)}$ is of size four. Then the graph $\overline G(Q_{\pi}(U, V))$ is formed by merging the edges of the unique path that exists between $U$ and $V$ in the tree $\overline{G}(\pi)$. Moreover, $\overline{Q_{\pi}(U, V)}$ is obtained from $\overline\pi$ by merging two blocks.
    \item When $k \neq 2$, this operation creates exactly one cycle of length $k$ in $\overline G(Q_{\pi}(U, V))$. Then $\overline G(Q_{\pi}(U, V))$ is formed by closing the unique path that exists between $U$ and $V$ in the tree $\overline{G}(\pi)$. Since no edges are merged,  $\overline{\pi} = \overline{Q_{\pi}(U, V)}$.
\end{itemize}
Hence, $Q_{\pi}(U, V) \in \U_{0,k}(n)$.
\end{remark}

\begin{definition}
We define the mapping
\[
    Q : \set{(\pi, U, V) : \pi \in \DT(n), U, V \in \pi, U \neq V} \to \bigcup_{k \geq 1} \U_{0,k}(n)
\]
given by $(\pi, U, V) \mapsto Q_{\pi}(U, V)$.
\end{definition}

Note that, by Remark~\ref{rmk:mapping_Q}, the mapping $Q$ is well-defined.

\begin{lemma}\label{lemma:smallest_partition}
Suppose $\rho \in \NC_2(n)$. 
\begin{enumerate}
    \item If $\pi = \Kr^{-1}(\rho)$, then $\overline\pi = \rho$ and the paired edges of $G(\pi)$ have opposite orientations. 

    \item If $\pi \in \PP(n)$ is such that $\overline\pi = \rho$ and the paired edges of $G(\pi)$ have opposite orientations, then $\Kr^{-1}(\rho) \leq \pi$.
\end{enumerate}
\end{lemma}

\begin{proof}
By Lemma \ref{lemma:description-double-trees}, we know that $\pi = \Kr^{-1}(\rho) \in \DT(n)$ with $\overline{\pi} = \rho$. Moreover, by the definition of a double-tree, the paired edges of $G(\pi)$ in $\rho$ have opposite orientations. This proves $(i)$.

Let $\pi \in \PP(n)$ be such that $\overline{\pi} = \rho$ and the paired edges in $G(\pi)$ have opposite orientations. Let $V = \{i_1, \dots, i_k\} \in \Kr^{-1}(\rho)$ be a vertex of $G(\Kr^{-1}(\rho))$. Let $U \in G(\pi)$ be the vertex containing $i_1$. We shall show that $V \subseteq U$. As $V$ was an arbitrary vertex of $G(\Kr^{-1}(\rho))$, this will prove $(ii)$. If $k = 1$ then $V = \{i_1\} \subseteq U$ and we are done. So now suppose $k \geq 2$.

\begin{figure}
\begin{center}
    \includegraphics[width=12em]{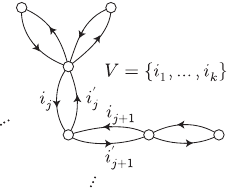}
    \caption{\label{fig:grho}\small The graph of $G(\rho)$ in the proof of Lemma \ref{lemma:smallest_partition}.}
\end{center}
    
\end{figure}
We shall suppose that we have labelled the edges of $G(\Kr^{-1}(\rho))$ pointing to $V$ in order so that:
\[
i_1 < i_2 < \cdots < i_k.
\]
See Figure \ref{fig:grho}. We move around the graph $G(\Kr^{-1}(\rho))$, starting with the edge labelled $i_1$; then move to $\gamma(i_1)$, then to $\gamma^2(i_1)$, and then to $\gamma^3(i_1)$ until we return to $i_1 = \gamma^n(i_1)$ after $n$ steps. While traversing this circuit we must visit the edges of $G(\Kr^{-1}(\rho))$ in the order of the labels. 

By one hand, since $G(\Kr^{-1}(\rho))$ is a double-tree, for each $1 \leq j \leq k$, there is $i'_j$ such that $\{i_j, i'_j\} \in \rho$. Moreover, since $G(\Kr^{-1}(\rho))$ is the quotient of $G(\gamma)$ with respect to $\Kr^{-1}(\pi)$, it follows that the first time we return to the vertex $V$ after moving from the edge labelled $i_j$, when traversing the edges of $G(\Kr^{-1}(\rho))$ in the order of the labels, is trough the edge labelled $i_j'$.

By the other hand, by construction of $G(\pi)$, the edge $(\gamma(i'_1), i_1')$ labelled $i'_1$ has a vertex in $U = [\gamma(i'_1)]_{\rho}$. That is, because of the way we labelled $\set{i_1, \ldots, i_l}$ and the fact that $i_1 < \gamma(i_1')$, it follows that $i_2 = \gamma(i_1') \in U$. By the same argument, we have that $i_3 = \gamma(i'_2) \in U$. Continuing inductively through the list $\{i_1, \dots, i_k\}$, we conclude that $V = \{i_1, \dots, i_k\} \subseteq U$, as claimed. 
%
\end{proof}

\begin{lemma}\label{lemma: merge vertices}
Let $\pi \in \U_{0,k}(n)$, for some $k \geq 1$. Let $Q^{-1}(\pi)$ be the set of triples $(\rho, U, V)$ such that $Q_{\rho}(U, V) = \pi$. If $k \neq 2$, then $\abs{Q^{-1}(\pi)} = 1$. If $k = 2$, then $\abs{Q^{-1}(\pi)} = 2$.
\end{lemma}

\begin{proof}
Suppose $k \neq 2$. By Corollary \ref{cor:m-k unicycle}, it follows that $\overline{\pi} \in \NC_2(n)$. Let $\rho = \Kr^{-1}(\overline{\pi})$. Since $\Kr^{-1} : \NC_2(n) \to \PP(n)$ is a one-to-one function, by Lemma \ref{lemma:description-double-trees}, it follows that $G(\rho)$ is the only double-tree such that
\[
    \overline{\rho} = \Kr(\Kr^{-1}(\overline{\pi})) = \overline{\pi}.
\]
Moreover, we know that $\#(\rho) = \frac{n}{2} + 1$. Because of Lemma \ref{lemma:smallest_partition}, we obtain that $\rho \leq \pi$. Since $\#(\pi) = \frac{n}{2}$, this shows the existence of unique $U, V \in \rho$ such that $Q_{\rho}(U, V) = \pi$.

Suppose now that $k = 2$. 
By Corollary \ref{cor:m-k unicycle}, and the definition of a $(0,2)$-unicycle, we know that the blocks of $\overline\pi \in \NC(n)$ are of size two, except for one block of size four. Then, there exist exactly two partitions $\rho_1, \rho_2 \in \NC_2(n)$ such that $\rho_1, \rho_2 \leq \overline\pi$. Because of Lemma \ref{lemma:description-double-trees} and the fact that $\Kr^{-1} : \NC(n) \to \NC(n)$ is a bijection, we know that $\pi_i = \Kr^{-1}(\rho_i) \in \DT(n)$ with $\overline\pi_i = \rho_i$ and $\pi_1 \neq \pi_2$. Again, because of Lemma \ref{lemma:smallest_partition} (considering $G(\pi)$ with $\pi_j$ as the paired edges in $G(\pi)$), it follows that $\pi_1, \pi_2 \leq \pi$. Since $\#(\pi) = \frac{n}{2}$, this shows the existence of unique $U, V \in \pi_i$ such that $Q_{\pi_i}(U, V) = \pi$. This is what was claimed.
\end{proof}

\begin{remark}
Note that $\overline\pi_1$ and $\overline\pi_2$ are obtained from $\pi$ by splitting the edges in the block of size four into two pairs of edges with opposite orientations.
\end{remark}

Recall that $B_n = \sum_{\#(\pi) = n/2} E(\pi, \bm j, \bm \epsilon)$ where $\pi \in \PP(n)$. By Equation \eqref{eq:edge upper bound} we have that either $\#(\overline\pi) = n/2 - 1$ or $\#(\overline\pi) = n/2$. So by Lemmas \ref{lemma:2-4 trees} and \ref{lemma:m-k unicycle} respectively, we can write
\begin{align*}
    B_n &= \kern-0.75em \sum_{\pi \in \U_{0,1}(n)} \tilde{f}(\pi, \bm \epsilon) \1_{\overline\pi \leq \ker(\bm j)} + \sum_{\pi \in \U_{0,2}(n)} E(\pi, \bm j, \bm \epsilon)\\
    & + \sum_{\substack{k \geq 3,\\ \pi \in \U_{0,k}(n)}} f(\pi, \bm \epsilon) \1_{\overline\pi \leq \ker(\bm j)} + \sum_{\substack{k \geq 3\\ \pi \in \U_{1,k}(n)}} g(\pi, \bm \epsilon)\1_{\overline\pi \leq \ker(\bm j)}.
\end{align*}

Recall that $\#(\pi) = \frac{n}{2} + 1$ for all $\pi \in \DT(n)$, and then there are $\binom{\frac{n}{2} + 1}{2}$ ways of choosing two different blocks of $\pi$. Moreover, for $\pi \in \U_{0,k}(n)$ with $k \neq 2$, we know that $\overline{\pi} = \overline{Q^{-1}(\pi)}$. For any $\pi \in \U_{0,2}(n)$, let $\set{\pi_1, \pi_2} = Q^{-1}(\pi)$ be the preimage provided by Lemma \ref{lemma: merge vertices}. For $\pi \in \U_{0,1}(n)$, we decree (by convention) that the edges in the loop have opposite orientation. With this convention $f$ is also defined on $\U_{0, 1}(n)$. We extend the definition of $f(\pi, \bm \epsilon)$ for $\pi \in \U_{0,2}(n)$ by considering $\overline\pi_1$ or $\overline\pi_2$ as the pairing on the edges of $G(\pi)$. We implicitly state which pairing we are using based on the context. Then, by Lemma \ref{lemma: merge vertices},
\begin{align*}
    \binom{\frac{n}{2} + 1}{2}A_n &= \binom{\frac{n}{2} + 1}{2}  \sum_{\pi \in \DT(n)} f(\pi, \bm \epsilon) \1_{\overline{\pi} \leq \ker(\bm j)} = \sum_{\substack{\pi \in \DT(n)\\ U \not= V \in \pi}} f(\pi, \bm \epsilon) \1_{\overline\pi \leq \ker(\bm j)}\\
    &= \kern-0.75em \sum_{\substack{k \neq 2\\ \pi \in \U_{0,k}(n)}} f(\pi, \bm \epsilon)\1_{\overline\pi\leq \ker(\bm j)} + \sum_{\substack{\pi \in \U_{0,2}(n)}} \Big\{f(\pi, \bm \epsilon) \1_{\overline{\pi}_1 \leq \ker(\bm j)} \\&\qquad  + f(\pi, \bm \epsilon) \1_{\overline{\pi}_2 \leq \ker(\bm j)}\Big\}.
\end{align*}
Therefore,
\begin{align}\label{eq:complex 1/N expansion}
    B_n & - \binom{\frac{n}{2} + 1}{2}A_n 
    = \kern-0.75em \sum_{\substack{k \geq 3\\ \pi \in \U_{1,k}(n)}} g(\pi, \bm \epsilon) \1_{\overline \pi \leq \ker(\bm j)} + \sum_{\pi \in \U_{0,1}} (\tilde{f}(\pi, \bm \epsilon) - f(\pi, \bm \epsilon))\1_{\overline\pi \leq \ker(\bm j)} \notag\\
    &\kern -1 em + \sum_{\pi \in \U_{0,2}(n)} \paren*{E(\pi, \bm j, \bm \epsilon) - f(\pi, \epsilon)\1_{\overline\pi_1 \leq \ker(\bm j)} - f(\pi, \epsilon) \1_{\overline\pi_2 \leq \ker(\bm j)}}.
\end{align}

Recall from Lemma \ref{lemma: merge vertices} that, when $k = 2$, there were two pairings $\overline\pi_1$ and $\overline\pi_2$ in $\NC_2(n)$ with $\overline\pi_1, \overline\pi_2 \leq \overline\pi$. There is a third pairing $\overline\pi_3 \leq \overline\pi$, but this pairing has a crossing. This pairing is the one where we pair edges with the same orientation in the block of $\overline\pi$ of size 4. We have to include this in our calculation when $r^2 \neq 0$. 
Denote $\U_{1,k}(n) := \U_{0,k}(n)$ for $k = 1, 2$. We extend the definition of  $g(\pi, \bm \epsilon)$ for $\pi \in \U_{1,k}(n)$, with $k = 1, 2$, in the following way:
\begin{itemize}
    \item if $k = 1$, we consider that the loops of $G(\pi)$ have the same orientation;
    \item if $k = 2$, we consider $G(\pi)$ with the pairing $\overline{\pi}_3$ on its edges.
\end{itemize}
Using our new notation we have
\[
    \sum_{\substack{k = 1, 2\\ \pi \in \U_{0,k}(n)}} g(\pi, \bm \epsilon) \1_{\overline\pi \leq \ker(\bm j)} = \kern-0.75em\sum_{\substack{k = 1, 2\\ \pi \in \U_{1,k}(n)}} g(\pi, \bm \epsilon) \1_{\overline\pi \leq \ker(\bm j)},
\]
so (with the convention that $\overline\pi_3$ denotes the  pairing under $\overline\pi$ that has a crossing in the block of size $4$, $\overline\pi_1$ and $\overline\pi_2$ are the non-crossing ones)
\begin{align}\label{eq:real 1/N expansion}\lefteqn{
    B_n - \binom{\frac{n}{2} + 1}{2}A_n 
    = \sum_{\substack{k \geq 1\\ \pi \in \U_{1,k}(n)}} g(\pi, \bm \epsilon)  \1_{\overline \pi \leq \ker(\bm j)}} \notag \\ 
    &\mbox{}+ \sum_{\pi \in \U_{0,1}(n)} (\tilde f(\pi, \bm \epsilon) - f(\pi, \bm \epsilon) - g(\pi, \bm \epsilon))
    \1_{\overline\pi \leq \ker(\bm j)} \notag \\
    & +  \sum_{\pi \in \U_{0,2}(n)} \Big(E(\pi, \bm j, \bm \epsilon)- f(\pi, \bm \epsilon) \1_{\overline{\pi}_1 \leq \ker(\bm j)}\notag \\ &\qquad\mbox{} - f(\pi, \bm \epsilon)\1_{\overline{\pi}_2 \leq \ker(\bm j)} - g(\pi, \bm \epsilon)\1_{\overline\pi_3 \leq \ker(\bm j)}\Big).
\end{align}




\begin{remark}
When $W$ is a GUE matrix, we have that $r^2 = 0$, $s^2 - \sigma^2 - r^2 = 0$, and $k_4(W^{(1)}_{12}, \overline{W^{(1)}_{12}},\ab W^{(1)}_{12}, \overline{W^{(1)}_{12}}) = 0$. When $W$ is a GOE matrix, we have that $r^2 = \sigma^2$, $s^2 - \sigma^2 - r^2 = 0$ and $k_4(W^{(1)}_{12}, \overline{W^{(1)}_{12}},\ab W^{(1)}_{12}, \overline{W^{(1)}_{12}}) = 0$. We shall see in the next lemma that this implies that in both these cases the infinitesimal cumulants vanish.
\end{remark}

\begin{lemma}\label{lemma:non-crossing part}
Let 
\[
\kappa_2(w_{j_u}^{(\epsilon_u)}, w_{j_v}^{(\epsilon_v)}) = 
\esp\Big(W_{12}^{(j_u, \epsilon_u)} \overline{W_{12}^{(j_v, \epsilon_v)}}\,\Big)=
\1_{ j_u = j_v} 
\begin{cases}
\sigma^2   & \text{if }\epsilon_u = \epsilon_v, \\
r^2  & \text{if } \epsilon_u = - \epsilon_v,
\end{cases}
\]
\begin{align*}\lefteqn{
\kappa_2'( w_{j_u}^{(\epsilon_u)}, w_{j_v}^{(\epsilon_v)} ) } \\ 
& =
\esp\Big(W_{11}^{(j_u)} W_{11}^{(j_v)}\,\Big)
- 
\esp\Big(W_{12}^{(j_u, \epsilon_u)} \overline{W_{12}^{(j_v, \epsilon_v)}}\,\Big)
- 
\esp\Big(W_{12}^{(j_u, \epsilon_u)} \overline{W_{12}^{(j_v, -\epsilon_v)}}\,\Big),
\end{align*}
and 
\[
\kappa_4'(w_{j_t}^{(\epsilon_t)},  w_{j_u}^{(\epsilon_u)}, w_{j_v}^{(\epsilon_v)}, w_{j_y}^{(\epsilon_y)}) 
=
\re\Big[k_4(W^{(j_t, \epsilon_t)}_{12}, \overline{W^{(j_u, \epsilon_u)}_{12}}, W^{(j_v, \epsilon_v)}_{12}, \overline{W^{(j_y, \epsilon_y)}_{12}})\Big],
\]
and let all other $\kappa_n$'s and $\kappa_n'$'s be $0$. Then
\begin{align}\label{eq:block of size 4 complex}\lefteqn{
 \sum_{\pi \in \U_{0,1}(n)}  (\tilde{f}(\pi, \bm \epsilon) - f(\pi, \bm \epsilon) - g(\pi, \bm \epsilon))\1_{\overline\pi \leq \ker(\bm j)}} \notag\\
& \mbox{} +
\sum_{\pi \in \U_{0,2}(n)} 
\big(E(\pi, \bm j, \bm \epsilon) 
- 
f(\pi, \bm \epsilon)\1_{\overline\pi_1 \leq \ker(\bm j)} 
- 
f(\pi, \bm\epsilon) \1_{\overline\pi_2 \leq \ker(\bm j)} \notag \\
& \qquad\qquad\qquad\mbox{} -
g(\pi, \bm\epsilon) \1_{\overline\pi_3 \leq \ker(\bm j)} \big)\\ \notag
& =
\sum_{\pi \in \NC(n)} \partial \kappa_\pi(w_{j_1}^{(\epsilon_1)}, w_{j_2}^{(\epsilon_2)}, \dots, w_{j_n}^{(\epsilon_n)}).
\end{align}
\end{lemma}

\begin{proof}
We consider the equation \eqref{eq:real 1/N expansion}. In the sum 
\[
 \sum_{\pi \in \U_{0,1}(n)}  (\tilde{f}(\pi, \bm \epsilon) - f(\pi, \bm \epsilon) - g(\pi, \bm \epsilon))\1_{\overline\pi \leq \ker(\bm j)},
\]
we have that
\begin{align*}\lefteqn{
\tilde{f}(\pi, \bm \epsilon) - f(\pi, \bm \epsilon) - g(\pi, \bm \epsilon)  }\\
& = 
\bigg[\esp\Big(W_{11}^{(j_u, \epsilon_u)} W_{11}^{(j_v, \epsilon_v)}\,\Big) 
-
\esp\Big(W_{12}^{(j_u, \epsilon_u)} \overline{W_{12}^{(j_v, \epsilon_v)}}\,\Big)  \\
& \qquad\qquad -
\esp\Big(W_{12}^{(j_u, \epsilon_u)} \overline{W_{12}^{(j_v, -\epsilon_v)}}\,\Big)\bigg]
 r^{2p}\sigma^{2q},
\end{align*}
where $(u, v)$ are the edges of the loop in $G(\pi)$, $p$ is the number of pairs $(u', v')$ of $\overline\pi$ (not on the loop) such that $\epsilon_{u'} = - \epsilon_{v'}$, and $q$ is the number such that $\epsilon_{u'} = \epsilon_{v'}$. Then, we choose one block of $\overline\pi$, and replace $\sigma^2$ (or $r^2$) by 
\[
\esp\Big(W_{11}^{(j_u, \epsilon_u)} W_{11}^{(j_v, \epsilon_v)}\,\Big)  
-
\esp\Big(W_{12}^{(j_u, \epsilon_u)} \overline{W_{12}^{(j_v, \epsilon_v)}}\,\Big)
-
\esp\Big(W_{12}^{(j_u, \epsilon_u)} \overline{W_{12}^{(j_v, -\epsilon_v)}}\,\Big).
\]

From Lemma \ref{lemma: merge vertices}, we know that the graphs $G(\pi)$, with $\pi \in \U_{0, 1}(n)$, are parameterized by the set of non-crossing pairings $\overline \pi \in \NC_2(n)$ and a block of $\overline\pi$, where we make a loop by joining the endpoints of this marked block (which is an edge in $\overline{G}(\pi)$). Thus, on the loop part, we are obtaining $\kappa_2'(w_{j_u}^{(\epsilon_u)}, w_{j_v}^{(\epsilon_v)})$ and on all other edges $\kappa_2(w_{j_u}^{(\epsilon_u)}, w_{j_v}^{(\epsilon_v)})$, and we sum over all edges. This is the same as the Leibniz rule for computing $\partial \kappa_\rho(w_{j_1}, \dots, w_{j_n})$, where we choose one block of $\rho$ and replace $\kappa_2$ by $\kappa'_2$. Hence 
\[
 \sum_{\pi \in \U_{0,1}(n)}  (\tilde{f}(\pi, \bm \epsilon) - f(\pi, \bm \epsilon) - g(\pi, \bm \epsilon))\1_{\overline\pi \leq \ker(\bm j)} 
= 
\sum_{\rho \in \NC_2(n)} \partial \kappa_\rho(w_{j_1}^{(\epsilon_1)},  \dots, w_{j_n}^{(\epsilon_n)}).
\]

Now let us adopt some ad-hoc notation:
\[
\NC_{2,4}(n) = \{ \pi \in \NC(n) \mid \pi
\text{ has one block of size $4$ and all others of size $2$}\}.
\]
Given that $\kappa_n = 0$ unless $n = 2$, and $\kappa_n' = 0$ unless $n= 2$ or $n = 4$, we may write the right hand side of \eqref{eq:block of size 4 complex} as
\[
\sum_{\pi \in \NC_2(n)} \partial \kappa_\pi(w_{j_1}^{(\epsilon_1)}, w_{j_2}^{(\epsilon_2)}, \dots, w_{j_n}^{(\epsilon_n)})
+
\sum_{\pi \in \NC_{2,4}(n)} \partial \kappa_\pi(w_{j_1}^{(\epsilon_1)}, w_{j_2}^{(\epsilon_2)}, \dots, w_{j_n}^{(\epsilon_n)}).
\]
So we just have to show that the second term above is
\[
\sum_{\pi \in \U_{0,2}(n)} \big(E(\pi, \bm j, \bm \epsilon)  
- f(\pi, \epsilon) \1_{\overline\pi_1 \leq \ker(\bm j)} 
- f(\pi, \epsilon) \1_{\overline\pi_2 \leq \ker(\bm j)}
- g(\pi, \epsilon) \1_{\overline\pi_3 \leq \ker(\bm j)}
\Big).
\]

If $\pi \in \U_{0,2}(n)$, then $\overline\pi$ has one block of size $4$ and all others of size $2$. Suppose the block of $\overline\pi$ of size $4$ is $\{ l_1, l_2, l_3, l_4\}$ with $l_1 < l_2 < l_3 < l_4$. This block of size $4$ has to split into $2$ blocks of size $2$ in each of $\overline\pi_1$, $\overline\pi_2$ and $\overline\pi_3$. In the first two there is no crossing while in the third one there is a crossing. Let us say that for $\overline\pi_1$ the two blocks are $\{ l_1, l_2\}$ and $\{l_3, l_4\}$; for $\overline\pi_2$ the two blocks are $\{l_1, l_4\}$ and $\{l_2, l_3\}$; and for $\overline\pi_3$ the two blocks are $\set{l_1, l_3}$ and $\set{l_2, l_4}$. The contribution of this block to $E(\pi, \bm j, \bm \epsilon)$ is
\[
E\Big(
    W^{(j_{l_1},\epsilon_{l_1})}_{i_{l_1} i_{{l_1}+1}}
\overline{W^{(j_{l_2},\epsilon_{l_2})}_{i_{{l_2}} i_{{l_2}+1}}}
W^{(j_{l_3},\epsilon_{l_3})}_{i_{{l_3}} i_{{l_3}+1}} 
\overline{W^{(j_{l_4},\epsilon_{l_4})}_{i_{{l_4}} i_{{l_4}+1}}} \,\Big).
\]
\begin{figure}[t]
\begin{center}
\noindent
\includegraphics[width=13em]{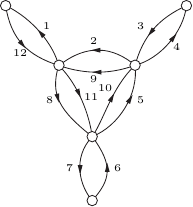}\hfill\includegraphics[width=15em]{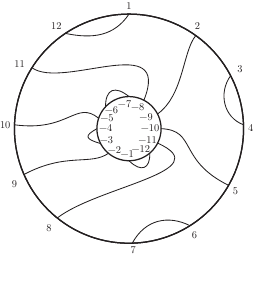}
\end{center}
\caption{\small\raggedright\label{fig:symmetric_annular_example}%
On the left we have $G(\pi)$ with 
$\pi = \{(1),\ab (2,  \ab 9, \ab 12), \ab  (3, 5, 10), (4), (6, 8, 11), (7)\}$, and 
$\overline\pi = \{(1, 12), (2, 9),\ab (3, 4),\ab (5, 10),  (6, 7), (8, 11)\}$. On the right we have $\sigma = \{(1,  12), \ab (2, -9),\ab (3, 4), (5, -10), (6, 7),(8, -11), (-1, -12), (-2 , 9),(-3, -4), \ab (-5 ,10), \ab(-6, -7), (-8, 11)\}$.  Note that for $(r, s) \in \overline \pi$, with the edges labelled $r$ and $s$ having the same orientation, we get $(r, -s)$ and $(-r, s)$ in $\sigma$. When the edges labelled $r$ and $s$ have the same orientation we get $(r, s)$ and $(-r, -s)$ in $\sigma$.} 
\end{figure}
On the other hand the contribution of this block of size $4$ to 
\[
  f(\pi, \epsilon) \1_{\overline\pi_1 \leq \ker(\bm j)} 
+ f(\pi, \epsilon) \1_{\overline\pi_2 \leq \ker(\bm j)}
+ g(\pi, \epsilon) \1_{\overline\pi_3 \leq \ker(\bm j)}
\]
is
\[
E\big(
          W^{(j_{l_1},\epsilon_{l_1})}_{i_{{l_1}} i_{{l_1}+1}} 
\overline{W^{(j_{l_2}, \epsilon_{l_2})}_{i_{{l_2}} i_{{l_2}+1}}} \big)
E\big(
          W^{(j_{l_3}, \epsilon_{l_3})}_{i_{{l_3}} i_{{l_3}+1}} 
\overline{W^{(j_{l_4}, \epsilon_{l_4})}_{i_{{l_4}} i_{{l_4}+1}}}\big)
\]
\[ \mbox{}\, + 
E\big(
          W^{(j_{l_1}, \epsilon_{l_1})}_{i_{{l_1}} i_{{l_1}+1}} 
\overline{W^{(j_{l_4}, \epsilon_{l_4})}_{i_{{l_4}} i_{{l_4}+1}}} \big)
E\big(
\overline{W^{(j_{l_2}, \epsilon_{l_2})}_{i_{{l_2}} i_{{l_2}+1}} }
          W^{(j_{l_3}, \epsilon_{l_3})}_{i_{{l_3}} i_{{l_3}+1}}\big) 
\]
\[\mbox{}\, + 
E\big(
          W^{(j_{l_1}, \epsilon_{l_1})}_{i_{{l_1}} i_{{l_1}+1}} 
\overline{W^{(j_{l_3}, \epsilon_{l_3})}_{{i_{l_3}} i_{{l_3}+1}}} \big)
E\big(
\overline{W^{(j_{l_2}, \epsilon_{l_2})}_{i_{{l_2}} i_{{l_2}+1}} }
          W^{(j_{l_4}, \epsilon_{l_4})}_{i_{{l_4}} i_{{l_4}+1}}\big). 
\]
Note that because our matrix entries are centred these terms will vanish unless $\1_{\overline\pi_1 \leq \ker(\bm j)}$ and $\1_{\overline\pi_2 \leq \ker(\bm j)}$.
Thus the contribution of the difference is exactly
\[
k_4\big(
          W^{(j_{l_1}, \epsilon_{l_1})}_{i_{{l_1}} i_{{l_1}+1}}, 
\overline{W^{(j_{l_2}, \epsilon_{l_2})}_{{i_{l_2}} i_{{l_2}+1}}},
          W^{(j_{l_3}, \epsilon_{l_3})}_{{i_{l_3}} i_{{l_3}+1}}, 
\overline{W^{(j_{l_4}, \epsilon_{l_4})}_{i_{{l_4}} i_{{l_4}+1}}}\big),
\]
which will vanish unless $j_{l_1} = j_{l_2} = j_{l_3} = j_{l_4}$. We then have to average this with its complex conjugate (as we might have $i_{l_1} < i_{l_1 + 1}$ or the other way around). In either case, the contribution is $\kappa_4'(w_{j_{l_1}}^{(\epsilon_{l_1})}, w_{j_{l_2}}^{(\epsilon_{l_2})}, w_{j_{l_3}}^{(\epsilon_{l_3})}, w_{j_{l_4}}^{(\epsilon_{l_4})})$.  The contribution of all the other edges is exactly as in the first part of the proof.
\end{proof}

\section{The bijection with symmetric non-crossing annular permutations}

Lemma \ref{lemma:non-crossing part} identifies the graphs $\U_{0,k}(n)$,  which depend on the infinitesimal cumulants. In this section we describe the contribution of the non-orientable graphs $\U_{1, k}(n)$ in terms of non-crossing annular pairings. 

\begin{notation}\label{notation:epsilon construction}
Let $\delta(k) = -k$ for $k \in [ \pm n]$. We let $S_{\pm n}$ be the permutations of $[ \pm n] = \{1, \dots, n\} \cup \{-1, \dots, -n\}$. We let $\Z_2 = \{-1, 1\}$, and given $\bm\epsilon = (\epsilon_1, \dots, \epsilon_n) \in \Z_2^n$, we construct a permutation in $S_{\pm n}$ (also denoted $\bm\epsilon$) by setting
\[
\epsilon(k) = 
\begin{cases} k & \text{if } \epsilon_{|k|} =   1, \\
             -k & \text{if } \epsilon_{|k|} = - 1.
\end{cases}
\]
Then $\epsilon^2$ is the identity.

We shall also follow the notation employed in \cite{m,mvb} (and elsewhere) regarding the embedding of $S_n$ into $S_{\pm n}$. Given a permutation $\tau \in S_n$, we consider it to be a permutation of $[\pm n]$ by letting it act trivially on $[-n ] = \{-1, \dots, -n\}$, i.e. $\tau(-k) = -k $ for $k \in [n]$.

Let $\rho \in \PP_2(n)$ be a pairing and $\epsilon \in \Z_2^n$. Using our convention above, let
\[
\sigma = \epsilon \rho \delta \rho \epsilon.
\]
Then $\sigma$ is a pairing, $\delta\sigma = \sigma\delta$, and $\sigma\delta$ is a pairing.
\end{notation}

\begin{lemma}\label{lemma:symmetric annular case}
Suppose $\pi \in \U_{1,k}(n)$ with $k \geq 1$. 
Let $\sigma \in S_{\pm n}$ be the permutation with cycles $(r, s)(-r, -s)$ whenever $(r, s) \in \overline\pi$ and the edges in $G(\pi)$ have opposite orientations; and  $(r, -s)(-r, s)$ whenever $(r, s) \in \overline\pi$ and the edges in $G(\pi)$ have the same  orientations. Then $\sigma \in \NC_2^\delta(n, -n)$. Conversely, every element of $\NC_2^\delta(n, -n)$ arises from a $\pi \in \U_{1, k}(n)$ in this way.
\end{lemma}

\begin{proof}
Suppose $\pi \in \U_{1, k}(n)$. From $\overline\pi$, we construct a permutation $\epsilon \in \Z_2^n \subseteq S_{\pm n}$ as follows. For each edge $(u, v) \in \overline\pi$, with $u < v$, we let $\epsilon_u = 1$; and $\epsilon_v = 1$ if $u$ and $v$ have the same orientation, and $\epsilon_v = -1$ if $u$ and $v$ have opposite orientations. Then, the $\sigma$ in the statement of the Lemma is $\epsilon \overline\pi \delta \overline\pi \epsilon$. See Figure \ref{fig:symmetric_annular_example} for an example. As mentioned in Notation \ref{notation:epsilon construction}, $\sigma$ is a pairing, so $\#(\sigma) = n$. To show that $\sigma \in \NC_2^\delta(n, -n)$ we need to show (see \cite[p. 5]{mvb}) that $\delta \sigma$ is a pairing and 
\begin{equation}
\#(\delta \gamma^{-1} \delta \sigma^{-1} \gamma) = n.
\end{equation}

\setbox1=\hbox{\includegraphics[width=13em]{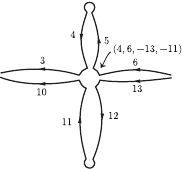}}
\setbox2=\hbox{\includegraphics[width=13em]{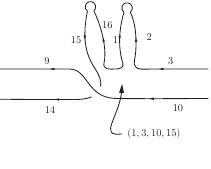}}
\begin{figure}[t]
\begin{center}
$\vcenter{\hsize=\wd1\box1}$
$\vcenter{\hsize=\wd2\box2}$
\end{center}
\caption{\label{fig:cycle vertex}\small Left: a vertex along a cycle. Right: a vertex along a cycle at the cross-over vertex. At this cross-over vertex the cycle is $(17, 1, 3, -10)$. In each case we show the corresponding cycle of $\tau$.}
\end{figure}

To get a conceptual idea of the proof, begin by supposing that $\overline G(\pi)$ is a cycle of length $k = n/2$. Then, as seen in Lemma \ref{lemma:double-cycle},
$$
\overline\pi = \{(1, k+1), (2, k+2), \dots, (k, 2k)\}.
$$
Then, $\sigma = \overline\pi \delta \overline\pi$ and $\gamma^{-1} \delta \sigma^{-1}\gamma = \delta\sigma$. Thus, $\#(\sigma) = n$ and
$$
\#(\delta\gamma^{-1}\delta \sigma^{-1} \gamma) = \#(\delta^2 \sigma) = \#(\sigma) = n.
$$
Hence,
\[
\#(\sigma) + \#(\delta \gamma^{-1} \delta \sigma \gamma)  = 2n,
\]
which implies that $\sigma$ is a pairing in $S_{NC}^\delta(n, -n)$, i.e. $\sigma \in \NC_2^\delta(n, -n)$, \cite[p. 5]{m}.

Next, we consider the general case: $G(\pi)$ has one cycle of length $k$ ($k \geq 1$), where in the cycle of length $k$ the edges are paired with edges with the same orientation. These form the through strings (or spokes) of $\sigma$, as in the special case above. The other parts of the graph are trees attached to the cycle; these form non-crossing pairings in the gaps between the through strings. See Figure \ref{fig:symmetric_annular_example}.



Let us label the edges of the cycle of length $k$ as $(u_1, v_1)$, \dots $(u_k, v_k)$. We order the edges $u_1 < \cdots < u_k$ and $u_j < v_l$ for $1 \leq l \leq k$. Since $u_l$ and $v_l$ have the same orientation, we must also have $v_1 < \cdots < v_k$. In addition, these intervals cannot overlap, so we may assume that 
\[
u_1 < \cdots < u_k < v_1 < \cdots < v_k.
\]
We call $\{u_1, \dots, u_k\}$ the \textit{outside part} of the cycle and $\{v_1, \dots, v_k\}$ the \textit{inside part} of the cycle. See Figure \ref{fig:moebius strip}.

\begin{figure}[t]
\begin{center}\includegraphics[width=15em]{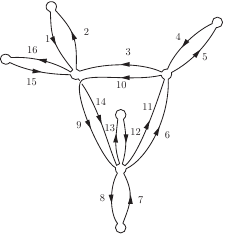}\end{center}
\caption{\label{fig:moebius strip} For $\pi \in \U_{1, k}(n)$, we traverse the edge of a M\"obius strip. Note the twist in the upper left.}
\end{figure}

Let 
\[
\tau = \eta \pi \delta \pi^{-1} \delta \eta,
\]
where we make the partition $\pi$ into a \textit{permutation} by specifying an order for the elements in each of its blocks (described below). In addition $\eta$ will be an element of $\Z_2^n$ considered as a permutation in $S_{\pm n}$ as described in Notation \ref{notation:epsilon construction}. Each cycle of $\pi$ will appear twice, once in each direction. So conjugation by $\eta$ will put minus signs in front of some elements. See Figure \ref{fig:cycle vertex}. 

There are two rules for putting the elements of each block of $\pi$ in the correct order. The simplest case is when a block of $\pi$ corresponds to a vertex of $G(\pi)$ not on the cycle; then we put the elements of this block in increasing order. We set $\eta_k = 1$ for a $k$ in one of these blocks.  The second case is when a vertex $V$ is on the cycle; then we put the elements of the block in order so the cycle of $\pi$ is $(i_1, \dots, i_l, j_m , \dots,  j_1)$ where $i_1 < \cdots < i_l < j_1 < \cdots < j_m$, in cyclic order.  In addition, for these elements we set $\eta_{i_s} = 1$ and $\eta_{j_s} = -1$. The corresponding cycles of $\tau$ will be
\[
(i_1, \dots, i_l, -j_m , \dots,  -j_1)
(j_1, \dots, j_m, -i_l , \dots,  -i_1)
\]
The sets $\{i_1, \dots, i_l\}$ and $\{j_1, \dots, j_m\}$ are specified as follows.

Let $t$ be such that the edges $u_t$ and $v_t$ are the incoming edges to the vertex $V$ (which we are assuming to be on the cycle). Then let $u_{t-1} < i_1 < \cdots < i_l = u_t$ be the labels of the incoming edges between $u_{t-1}$ and $u_t$. Let $v_t = j_1 < \cdots < j_m < v_{t-1}$ be the labels of the incoming edges between $v_t$ and $v_{t-1}$. In Figure \ref{fig:cycle vertex} (left) above we have $u_t = 6$ and $v_t = 13$. Then $i_1 = 4$ and $i_2 = 6$; also $j_1 = 10$ and $j_2 = 13$. Then $i_1 < i_2 < j_1 < j_2$ and the cycle of $\tau$ is $(i_1, i_2, -j_2, -j_1) = (4, 6, -13, -10)$.

\begin{figure}[t]
\begin{center}\includegraphics[width=15em]{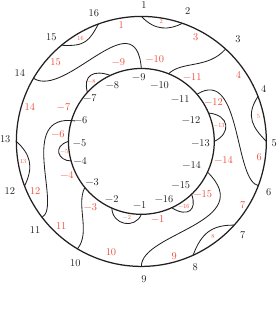}\end{center}
\caption{\label{fig:the corresponding sigma and tau}
The $\sigma$ corresponding to Figure \ref{fig:moebius strip} is displayed here, and the corresponding $\tau$ is in red.}
\end{figure}

At one special vertex $V$, which we call the \textit{cross-over vertex},
something slightly different happens. Again let $u_t$ and $v_t$ be the incoming edges ($3$ and $10$ in Figure \ref{fig:cycle vertex} (right)). If we follow the edge labelled $u_t$ backwards using $\gamma^{-1}$, we come to $v_{t-1}$, i.e. we have \textit{crossed} from the outside to the inside. In this case we let $v_{t-1} < i_1 < \cdots < i_l = u_t$ be the labels of the incoming edges between $v_{t-1}$ and $u_t$, and $v_t = j_1 < \cdots < j_m < u_{t-1}$ be the labels of the incoming edges between $u_{t-1}$ and $v_{t}$. This gives the construction of $\tau$.

By construction, the cycles of $\tau$ are the blocks of $\pi$, after the removal of the minus signs. Let us show that 
\[
\tau = \gamma \sigma^{-1} \delta\gamma^{-1}\delta.
\]
Since $\#(\eta\pi \delta \pi^{-1}\delta \eta) = \#(\pi \delta \pi^{-1}\delta) = 2 \#(\pi) = n$, this will then show that $\#(\delta\gamma^{-1}\gamma \sigma\gamma) = n$, and thus that $\sigma \in \NC_2^\delta(n, -n)$. 

Let us start with a vertex $V$ on the cycle, which may or may  not be the cross-over vertex. Let  $(i_1, \dots, i_l, -j_m,\ab \dots, -j_1)$ be the corresponding cycle of $\tau$.  Let $1 \leq s \leq l$. Then, $\delta\gamma^{-1}\delta(i_s) = i_s$, $\sigma(i_s)$ is the first outgoing edge moving clockwise from $i_s$ and $\gamma(\sigma(i_s))$ is the first incoming edge following $i_s$, i.e. for $s < l$, $\gamma\sigma\delta\gamma^{-1}\delta(i_s) = i_{s+1} = \tau(i_s)$. When $s = l$, $\sigma(i_l) = -j_m$ and $\gamma(\sigma(i_l)) = -j_m$. Thus  $\gamma\sigma\delta\gamma^{-1}\delta(i_l) = -j_m$. On the other hand, $\tau(i_l) =  - j_m$. Thus, on the set $\{i_1, \dots, i_l\}$, we have that $\gamma\sigma\delta\gamma^{-1}\delta$ and $\tau$ agree.

Continuing with the same vertex $V$, let us consider the action on $\{-j_m,\ab \dots,\ab -j_1\}$. We have $\delta\gamma^{-1}\delta(-j_s)$ is the label of the outgoing edge immediately following the edge labelled $j_s$ and $\sigma\delta\gamma^{-1}\delta(-j_s)$ is the label of the incoming edge immediately following the edge labelled $j_s$. This edge is labelled $j_{s-1}$ (for $s < m$ and $i_1$ when $s = m$) . Thus, when $s < m$,  $\sigma\delta\gamma^{-1}\delta(-j_s) = -j_{s-1}$ and $\gamma(-j_{s-1}) = -j_{s-1}$.  Thus $\gamma\sigma\delta\gamma^{-1}\delta(-j_s) = \tau(-j_s)$ in both cases. Thus $\tau$ and $\gamma\sigma\delta\gamma^{-1}\delta$ agree on $\{-j_m, \dots, -j_1\}$. 

All the other cycles of $\tau$ are reversals of the ones verified above. This proves that $\tau = \gamma\sigma\delta\gamma^{-1}\delta$ as claimed. 

Now let us suppose $\sigma \in \NC_2^\delta(n, -n)$. We create $\epsilon \in \Z_2^n$ as follows. For $(u, v) \in \sigma$, we let $\epsilon_u = \epsilon_v$ if $u$ and $v$ have the same sign; and $\epsilon_u = - \epsilon_v$ if $u$ and $v$ have the opposite sign. This does not determine $\epsilon$ uniquely, but by the $\delta$-symmetry of $\sigma$, the product $\epsilon\sigma\epsilon$ is independent of all choices and indeed must be of the form $\sigma_0\delta\sigma_0\delta$ for a pairing $\sigma_0 \in \PP_2(n)$. Indeed, for each $(u, v)(-u, -v) \in \sigma$, we  have $(|u|, |v|) \in \sigma_0$. We let $\tau = \gamma \sigma \delta \gamma^{-1} \delta$. From $\tau$, we extract a partition $\tau_0 \in \PP(n)$ by choosing one representative of each conjugate pair and then removing any minus signs. Let $2k$ be the number of through strings of $\sigma$. We claim that $\overline\tau_0 = \sigma_0$ and $\tau_0 \in \U_{1, k}(n)$. 

To show that $\overline\tau_0 \in \U_{1, k}$ we must show that:
\begin{itemize}

\item
$\#(\tau_0) = \#(\overline\tau_0) = n/2$, and

\item
$\overline{G}(\tau_0)$ has a cycle of length $k$, on which the edges have the same orientation, and the edges outside the cycle have the opposite orientation.

\end{itemize}

By construction, $\#(\tau_0) = \#(\tau)/2 = n/2$. In addition, $\#(\sigma) + \#(\delta \gamma^{-1} \delta \sigma \gamma)\ab = 2n$ because $\sigma \in \NC_2^\delta(n, -n)$; and $\#(\sigma) = n$ because $\sigma$ is a pairing. Thus $\#(\tau) = n$. So we have to show that $\#(\overline\tau_0) = n/2$. If we can show that $\overline\tau_0 = \sigma_0$ then this will be done. 

First, we show that $(\bm a)$: $\sigma_0 \leq \overline\tau_0$ (as partitions). This will show that $\#(\overline\tau_0) \leq \#(\sigma_0) = n/2$. In addition, we have Euler's equation ($\#(\tau_0) \leq 1 + \#(\overline\tau_0)$), with equality only if $\overline{G}(\tau_0)$ is a tree. Thus $n/2 -1 \leq \#(\overline\tau_0) \leq n/2$ and we can only have $n/2 -1 = \#(\overline\tau_0)$ when $\overline{G}(\tau_0)$ is a tree. Thus we also need to prove $(\bm b)$: $\overline{G}(\tau_0)$ has a cycle of length $k$, the edges on the cycle have the same orientation and the edges outside the cycle have the opposite orientation.

First, we prove $(\bm a)$. Suppose $(u, v) \in \sigma_0$. Let us suppose first that $\epsilon_u = \epsilon_v$. Then $(u, v) \in \sigma$. We must show that $u \sim_{\tau_0} \gamma(v)$, i.e. the edge labelled $v$ starts at the block of $\tau_0$ containing $v$ and ends at the block of $\tau_0$ containing $u$. But if $\epsilon_u = \epsilon_v$ , then $\sigma(u) = v$; $\sigma \delta\gamma^{-1}\delta(u) = \sigma(u) = v$ and $\tau(u) = \gamma \sigma \delta \gamma^{-1} \delta(u) = \gamma(v)$ as claimed. Next suppose that $\epsilon_u = - \epsilon_v$. Then $(u, -v) \in \sigma$ and $(-u, v) \in \sigma$.  Then $\tau(u) = \gamma\sigma(u) = \gamma(-v) = -v$. Thus $u$ and $v$ end at the same vertex of $G(\tau_0)$.  Also $\tau(-\gamma(v)) = \gamma \sigma \delta \gamma^{-1} \delta( \delta \gamma(v)) = \gamma(\sigma(-v)) = \gamma(u)$. Thus $\gamma(u) \sim_{\tau_0} \gamma(v)$. Hence, the edges labelled $u$ and $v$ start and end at the same vertices of $G(\tau_0)$ and have the same orientation. This proves $(\bm a)$. In addition, this also shows that the $k$ through strings of $\sigma$ produce a cycle in $\overline G(\tau_0)$, as each through string connects the cycles of $\tau$ on either side. Hence, $\overline G(\tau_0)$ has a cycle of length $k$ where the edges have the same orientation. In a similar way, we prove that the edges outside the cycle have opposite orientation. This proves $(\bm b)$. Hence $\tau_0 \in \U_{1, k}(n)$ and $\overline \tau_0 = \sigma_0$. 

This completes the proof of the bijection.
\end{proof}

\begin{remark}
    Let $\sigma \in S_\NC^\delta(n, -n) \setminus \NC_2^{\delta}(-n, n)$. It is immediate, from the fact that $\set{w_j, w_j^t}_{j \in J}$ is a semi-circular family, that $\kappa_{\sigma/2} = 0$.
\end{remark}

\begingroup
\renewcommand{\thetheorem}{\ref{thm:main}}
\addtocounter{theorem}{-1}
\begin{theorem}
Let $\{ W_j\}_{j \in J}$ be independent $N \times N$ Wigner matrices each distributed as in Definition \ref{def:wigner matrix}. Then the set $\{ W_j\}_{j \in J}$ is asymptotically semi-circular and real infinitesimally free. When $r^2 = 0$, the set $\{ W_j, W_j^t \}_{j \in J}$ is asymptotically  free.  

For an individual operator $w_j$, we have that the real infinitesimal free cumulants of $w_j$ are given by
\begin{itemize}

\item
$\kappa_2( w_{j_u}^{(\epsilon_u)}, w_{j_v}^{(\epsilon_v)} ) = k_2(W_{12}^{(j_u, \epsilon_u)}, W_{12}^{(j_v, -\epsilon_v)})$,


\item
\begin{align*}\lefteqn{
\kappa_2'( w_{j_u}^{(\epsilon_u)}, w_{j_v}^{(\epsilon_v)} ) } \\ 
& =
\esp\Big(W_{11}^{(j_u)} W_{11}^{(j_v)}\,\Big)
- 
\esp\Big(W_{12}^{(j_u, \epsilon_u)} \overline{W_{12}^{(j_v, \epsilon_v)}}\,\Big)
- 
\esp\Big(W_{12}^{(j_u, \epsilon_u)} \overline{W_{12}^{(j_v, -\epsilon_v)}}\,\Big),
\end{align*}
and

\item
\[ \kern-0.25em
\kappa_4'(w_{j_t}^{(\epsilon_t)},  w_{j_u}^{(\epsilon_u)}, w_{j_v}^{(\epsilon_v)}, w_{j_y}^{(\epsilon_y)}) 
=
\re\Big[k_4(W^{(j_t, \epsilon_t)}_{12}, \overline{W^{(j_u, \epsilon_u)}_{12}}, W^{(j_v, \epsilon_v)}_{12}, \overline{W^{(j_y, \epsilon_y)}_{12}})\Big],
\]

\item
All other infinitesimal cumulants vanish and

\item
$\kappa_2(w, w) = \esp(|W_{12}|^2)$

\item
$\kappa_2(w, w^t) = \esp(W_{12}^2)$

\item
$\kappa_2'(w, w) = \kappa_2'(w, w^t) = \esp(W_{11}^2 - |W_{12}|^2 - W_{12}^2)$

\item
$\kappa_4'(w, w, w, w) = \esp(|W_{12}|^4) - 2 \esp(|W_{12}|^2)^2 -   \esp(W_{12}^2)^2 $

\end{itemize}
\end{theorem}
\endgroup

\begin{proof}
Recall that the $1/N$ term, $B_n - \binom{n/2 + 1}{2} A_n$, was written in two parts in Equation \ref{eq:real 1/N expansion}. In Lemma \ref{lemma:symmetric annular case} we showed that the first part was given by
\[
\sum_{\sigma \in S_\NC^\delta(n, -n)} \kappa_{\sigma/2}(w_{j_1}^{(\epsilon_1)}, \dots, w_{j_n}^{(\epsilon_n)}).
\]
In Lemma  \ref{lemma:non-crossing part} we showed that the second part was given by
\[
\sum_{\pi \in \NC(n)} \partial\kappa_\pi(w_{j_1}^{(\epsilon_1)}, \dots, w_{j_n}^{(\epsilon_n)}). 
\]
Combined with the moment-cumulant formula, Equation (\ref{eq:real moment cumulant}), this shows that the  joint cumulants of $\{ w_j, w^t_j \}_{j \in J}$ are given by those in the statement of Lemma \ref{lemma:non-crossing part}. In particular:
\begin{itemize}

\item
$\{ w_j, w^t_j \}_{j \in J}$ is a semi-circular family of centred elements; i.e. all cumulants $\kappa_n(w^{(\epsilon_1)}_{j_1}, \dots, w^{(\epsilon_n)}_{j_n}) = 0$ for $n \geq 3$ whatever the $\bm j$ or the $\bm\epsilon$.

\item
$\kappa_2( w_{j_1}^{(\epsilon_1)}, w_{j_2}^{(\epsilon_2)}) = 0$ whenever $j_1 \not = j_2$. This shows that $\{w_{j_1}, w_{j_1}^t\}$, \dots, $\{w_{j_n}, w_{j_n}^t\}$ are free whenever $j_1, \dots, j_n$ are distinct.

\item
All real mixed infinitesimal cumulants $\kappa_n'( w_{j_1}^{(\epsilon_1)}, \dots, w_{j_n}^{(\epsilon_n)}) = 0$ whenever $j_1$, \dots, $j_n$ are not all equal. This holds for all $n \geq 1$ and all $\bm \epsilon$. This now gives us asymptotic real infinitesimal freeness.

\item 
The values of the infinitesimal cumulants were proved in Lemma \ref{lemma:non-crossing part}.
\end{itemize}
\end{proof}



\section{Concluding Remarks}

\begin{remark}
    The results presented here can be readily extended, under mild conditions, to the case of complex Wigner matrices whose entries need not be identically distributed, provided that they share the same first four moments (as required in Definition~\ref{def:wigner matrix}). In a general setting, one typically requires the entries to have uniformly bounded moments of all orders to establish these results.

Indeed, Equation (\ref{eq:def_E_pi_j_eps}) shows that $E(\pi,\bm j, \bm \epsilon)$ only depends on the distribution of the entries of the $W_j$'s and has no $N$ dependence, and thus the determination of which terms survive in the large $N$ limit only depends on $\pi$. We saw in Lemmas \ref{lemma:double-trees}, \ref{lemma:2-4 trees}, \ref{lemma:m-k unicycle}, and \ref{lemma:non-crossing part} that the terms that survive in the large $N$ limit only depend on the first four joint moments of the entries of the $W_j$'s, or equivalently the first four joint cumulants of the entries of the $W_j$'s. This then implies that the limit joint distribution and the limit joint infinitesimal distribution only depend on the first four cumulants. This means that we can relax the assumptions of Theorem \ref{thm:main}. Indeed, we assumed that the entries of the $W_j$'s are identically distributed; but in fact we only need this up to the first four moments, and to determine the limit eigenvalue distribution we only need to consider the first two moments. The same applies for the asymptotic freeness conclusion; this only needs the vanishing of $k_2(W_{i_1, i_2}^{(j_1, \epsilon_1)}, W_{i_3, i_4}^{(j_2, \epsilon_2)})$ for $j_1 \not = j_2$. When it comes to asymptotic infinitesimal freeness only the cumulants in Lemma \ref{lemma:non-crossing part}, which are only up to order $4$,  are needed to show infinitesimal freeness. 
\end{remark}

\begin{remark}\label{rem:pst connection}
In \cite[Thm.~5.1]{pst} the joint infinitesimal law of a GUE random matrix and its transpose was given in terms of the complex infinitesimal cumulants of \cite{fn}, which we will denote here by $\{ \mathop{\kappa'}_n^{(c)} \}_n$. Recall that for an infinitesimal probability space $(\cA, \phi, \phi')$ the free cumulants and complex infinitesimal cumulants are given by the moment cumulant formulas:
\[
\phi(a_1 \cdots a_n) = \sum_{\pi\in \NC)n)} \kappa_\pi(a_1, \dots, a_n)
\]
and
\begin{equation}\label{eq:complex case}
\phi'(a_1 \cdots a_n) = \sum_{\pi\in \NC)n)} \partial\kappa^{(c)}_\pi(a_1, \dots, a_n).
\end{equation}
By \cite{pst} for the semi-circular family, $\{ w, w^t \}$, arising from the GUE, we have $\kappa'_n(w^{(\epsilon_1)}, \dots, w^{(\epsilon_{2m})})  =1 $ when $\epsilon_i + \epsilon_{i+m} = 0$ for $1 \leq i \leq m$ and $\kappa'_n(w^{(\epsilon_1)}, \dots, w^{(\epsilon_{2m})})  = 0$ for all other $\epsilon$'s. On the other hand, by our Theorem \ref{thm:main} all the real infinitesimal cumulants are $0$. The aim of this remark is to show that these two statements are equivalent.  

Since $\kappa_2(w, w) = \kappa_2(w^{t}, w^{t}) = 1$, and the complex infinitesimal cumulants $\mathop{\kappa'}^{(c)}_2(w, w) = \mathop{\kappa'}^{(c)}_2(w^{t}, w^{t}) = 0$, we conclude from (\ref{eq:complex case}) and \cite[Thm.~5.1]{pst} that the infinitesimal moment $\phi'(w^{(\epsilon_1)} \cdots w^{(\epsilon)})$ is the number of non-crossing partitions $\pi$ with the following special property: all blocks have an even number of elements and one block is marked, let's call this the \textit{twisted block}. The twisted block can be of any even size, all the other blocks must be of size $2$. If $(i_1, i_2)$ is an untwisted block of size $2$ then we must have $\epsilon_{i_1} = \epsilon_{i_2}$. If the twisted block is $(i_1, \dots, i_{2k})$ then we must have $\epsilon_{{i_j}} = - \epsilon_{i_{j+k}}$ for $1 \leq j \leq k$. See Figure \ref{fig:nine and ten} (\textit{left}) for an example. 
\begin{figure}[t]\hfill
\includegraphics{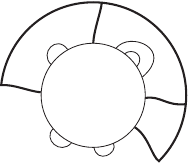}\hfill
\includegraphics{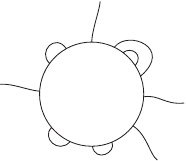}\hfill\hbox{}
\caption{\label{fig:nine and ten} On the left a non-crossing partition with the \textit{twisted block} in thick lines. On the right the `inside half'  of the corresponding element in $\NC_2^\delta(n, -n)$.}
\end{figure}
See Fig.~3 in \cite[p. 1000]{m} to see why we call such a block twisted. In \cite[Remark 21]{m} it is shown that these partitions are in bijection with $\NC_2^\delta(n, -n)$, with the twisted block of $\pi$ becoming the set of through strings of the corresponding element in $\NC_2^\delta(n, -n)$. Thus

\[
\phi'(w^{(\epsilon_1)} \cdots w^{(\epsilon)})
=
\sum_{\sigma \in \NC_2^\delta(n, -n)} \kappa_{\sigma/2}(w^{(\epsilon_1)}, \dots, w^{(\epsilon)}).
\]
Combining this with Equation (\ref{eq:real moment cumulant}) we have that all the real infinitesimal cumulants of $\{w, w^t\}$ are $0$, which is the conclusion of Theorem \ref{thm:main}.
\end{remark}

\begin{remark}
    In \cite{au}, Au shows asymptotic infinitesimal freeness of a Wigner matrix and fixed finite rank matrices. It seems very likely that the methods presented here can be extended to show asymptotic infinitesimal freeness of Wigner matrices and constant matrices. This has already been shown in \cite{cm} for orthogonally invariant and constant matrices.
\end{remark}

\begin{remark}
    It seems likely that the approach presented here should be adaptable to compute the infinitesimal distribution of other matrix models, such as $\beta$-ensembles and Wishart ensembles.
\end{remark}

\begin{acknowledgements*}
Research was supported by a Discovery Grant from the Natural Sciences and Engineering Research Council of Canada. Samuel Gurrola-Viramontes was supported by the Secretaría de Ciencia, Humanidades, Tecnología e Innovación (SECIHTI) through a graduate fellowship (Grant No. 1305033) and Apoyos Complementarios de Movilidad en el Extranjero, Movilidad Nacional, Movilidad en los Sectores de Interés y Movilidad para Programas de Doble Titulación 2026.
\end{acknowledgements*}

\noindent
{\scshape\fontsize{11}{13}\selectfont S.~Gurrola-Viramontes}, Centro de Investigación en Matemáticas, Guanajuato, Gto. 36000, Mexico, \hfill
\texttt{samuel.gurrola@cimat.mx}.

\medskip\noindent
{\scshape\fontsize{11}{13}\selectfont J.~A.~Mingo}, Department of Mathematics and Statistics, Queen's University, Kingston, ON, K7L, 3N6, Canada, \hfill
\texttt{james.mingo@queensu.ca}.

\thebottomline\end{document}